\documentclass[10pt,pra,letterpaper,superscriptaddress,twocolumn,showpacs, floatfix]{revtex4-2}
\pdfoutput=1

\usepackage[utf8]{inputenc}
\usepackage[T1]{fontenc}
\usepackage{amsmath, amssymb}
\usepackage{graphicx}
\usepackage{xcolor}
\usepackage{hyperref}
\hypersetup{pdfencoding=auto,psdextra=true}
\usepackage[linesnumbered, ruled, vlined]{algorithm2e}
\usepackage[compatibility=false]{caption}
\usepackage{subcaption}
\usepackage{ragged2e}
\usepackage{placeins}
\usepackage{tikz}
\usetikzlibrary{decorations.pathreplacing}
\usepackage{amsthm}
\newtheorem{observation}{Observation}
\newtheorem{theorem}{Theorem}
\newtheorem{lemma}{Lemma}
\newtheorem{corollary}{Corollary}
\newcommand{\mat}[1]{\mathsf{#1}}

\begin{document}
\title{Strongly Regular Graph Construction via Quadratic Binary Optimization}

\author{D. Osorio}
\affiliation{Instituto de F\'isica, Pontificia Universidad Cat\'olica de Chile, Casilla 306, Santiago, Chile}

\author{P. Andagana}
\affiliation{Instituto de F\'isica, Pontificia Universidad Cat\'olica de Chile, Casilla 306, Santiago, Chile}

\author{V. González}
\affiliation{Departamento de Física, Universidad de Antofagasta, Chile}

\author{J. Czartowski}
\affiliation{Jagiellonian University, Marian Smoluchowski Institute of Physics, Krak\'ow, Poland}
\affiliation{School of Physics, Trinity College Dublin, Dublin, Ireland}

\author{D. Goyeneche}
\affiliation{Instituto de F\'isica, Pontificia Universidad Cat\'olica de Chile, Casilla 306, Santiago, Chile}

\begin{abstract}
Strongly regular graphs, hereafter SRGs, are highly structured combinatorial
objects of central interest in algebraic graph theory and quantum information,
whose explicit construction becomes rapidly intractable as the number of
vertices grows; Conway's still unresolved $99$-graph problem is an emblematic
example. The matrix equation defining SRGs is quadratic in the binary edge
variables, so the natural penalty enforcing it, a sum of squared residuals, is
quartic and lies outside quadratic unconstrained binary optimization, known as
QUBO. We prove a considerably stronger obstruction, namely that no quadratic
function of the edge variables has the labeled SRGs as its exact set of global
minimizers, except for three degenerate families given by the perfect matchings,
the cocktail-party graphs and the pentagon. In particular, we formally establish
that the $99$-vertex graph of Conway's problem cannot be obtained from any QUBO
over the edge variables without auxiliary variables. Guided by this theorem we
develop two complementary algorithms. The first operates within pure QUBO,
solves the degenerate families exactly, and beyond them works as a
generate--prune--verify pipeline whose output is a reduced set of candidates
certified by a polynomial-time classical test. The second solves the full
problem by quadratizing this quartic penalty through Rosenberg auxiliary
variables, whose number grows as $\Theta(N^3)$ with the number of vertices
$N$. We test the pure-QUBO route under simulated annealing on all feasible
families with at most $18$ vertices and the auxiliary-variable route on
representative targets up to $16$ vertices, and tabulate the algebraically
feasible parameter sets for graphs with at most $49$ vertices.
\end{abstract}


\maketitle

% --- 1. Introducción ---
\section{Introduction}
\label{sec:intro}

Combinatorial optimization underlies a wide class of problems in mathematics, physics, and computer science, where optimal configurations must be selected from exponentially large discrete spaces \cite{handbook_combinatorial,modern_graph_theory}. A particularly rich and challenging class of such problems concerns the explicit construction of graphs with prescribed structural properties, since the constraints are typically nonlocal and induce highly non-convex search landscapes that scale rapidly with system size.

A graph with $N$ vertices is said to be $k$-regular if every vertex has degree $k$ \cite{west2001}; for fixed $k$, the number of non-isomorphic $k$-regular graphs grows superexponentially with $N$ \cite{mckay_wormald_1991}. Strongly regular graphs, hereafter SRGs, constitute a far more constrained family within this class. A graph is strongly regular with parameters $(N,k,\lambda,\mu)$ if it is $k$-regular, every pair of adjacent vertices has exactly $\lambda$ common neighbors, and every pair of non-adjacent vertices has exactly $\mu$ common neighbors \cite{brouwer_haemers_spectra}. All these conditions are captured by a single quadratic equation for the adjacency matrix $A$,
\begin{equation}\label{eq:srg-cert}
    A^2 \;=\; (\lambda-\mu)\,A \,+\, (k-\mu)\,\mat{I} \,+\, \mu\,\mat{J},
\end{equation}
together with $A_{ii}=0$ and $A\mathbf{1}=k\mathbf{1}$, where $\mat{I}$ is the identity matrix, $\mat{J}$ the all-ones matrix, and $\mathbf{1}$ the all-ones vector. The rigid spectral structure implied by Eq.~\eqref{eq:srg-cert}, exactly three distinct eigenvalues, makes SRGs natural objects in the study of adjacency matrices as Hamiltonians of discrete systems \cite{adjacency_hamiltonian} and in algebraic graph theory \cite{godsil_royle_algebraic}. They also appear in finite geometry and quantum information theory, in particular as the combinatorial substrate of mutually unbiased bases obtained from Paley graphs over prime power fields \cite{Bandyopadhyay2001} and of equiangular line sets in low-dimensional Hilbert spaces \cite{LEMMENS1973494}, where the spectrum of the SRG directly controls the inner-product structure of the associated quantum states. Deciding whether an SRG exists for given feasible parameters can be a profound problem. An emblematic case is Conway's $99$-graph problem \cite{conway2017five}, one of the five \$1000-prize problems posed by J.~H.~Conway, which asks whether an $\mathrm{SRG}(99,14,1,2)$ exists; the question remains open and has recently been attacked with SAT solvers \cite{keramatipour2026sat}.

Quadratic unconstrained binary optimization, known as QUBO, provides a unifying framework for combinatorial optimization in which a quadratic polynomial over binary variables is minimized \cite{qubo_review}. A key attraction of the framework is that a QUBO instance is solver-agnostic. The same instance can be attacked by classical simulated annealing \cite{simulated_annealing}, by quantum annealing \cite{quantum_annealing}, and, increasingly relevant in practice, by a new generation of high-performance classical and quantum-inspired QUBO engines such as VeloxQ \cite{pawlowski2025veloxq}, which routinely handle instances with millions of binary variables, as well as by tensor-network methods that simulate quantum optimization hardware at scale \cite{kshetrimayum2026tn}. Casting a combinatorial construction problem in QUBO form therefore plugs it into a rapidly advancing ecosystem of optimization technology, within which quantum annealers are one option among several rather than a prerequisite.

Is the SRG construction problem a QUBO? Encoding a graph by binary edge variables, the entries of $A^2$ are quadratic in these variables, so the natural penalty enforcing Eq.~\eqref{eq:srg-cert}, a sum of squared residuals, is a polynomial of degree four. The defining equation itself therefore cannot be minimized within QUBO; in the language of pseudo-Boolean optimization it yields a higher-order unconstrained binary optimization problem, HUBO for short \cite{boros_hammer2002,ishikawa2011}. This observation, however, does not close the question. Nothing forbids, \textit{a priori}, the existence of some other quadratic function, possibly far removed from the defining equation, whose global minimizers are exactly the desired SRG adjacency matrices; finding one would put the full problem within reach of every QUBO solver at no overhead. One of the central results of this work is that this possibility can be ruled out completely. We prove that a quadratic function of the edge variables whose global minimizers are exactly the labeled SRGs exists only for three degenerate families, namely the perfect matchings, the cocktail-party graphs and the pentagon, and that for them the elementary regularity penalty already does the job. For every other feasible parameter set, including the $\mathrm{SRG}(99,14,1,2)$ of Conway's problem, auxiliary variables are unavoidable. The proof rests on a short symmetry argument, an averaging over the vertex-relabeling group which shows that any candidate quadratic can be symmetrized, while symmetric quadratics can only measure degree statistics. Section~\ref{sec:noqubo} presents the result with its short proof, and Appendix~\ref{app:no-aux} develops the complete study, including the failure of invariant cubic penalties and an exhaustive analysis of symmetry-broken frames.

Guided by this theorem, we develop two complementary QUBO-based algorithms which, to the best of our knowledge, constitute the first QUBO treatments of SRG construction; previous QUBO encodings of graph-theoretic tasks have focused on maximum clique \cite{lucas2014,chapuis2019}, Ramsey numbers \cite{bian2014}, and related vertex-selection problems, none of which enforce the simultaneous constant-$\lambda$ and constant-$\mu$ conditions that define SRGs. The first algorithm, developed in Sec.~\ref{sec:pipeline}, stays within pure QUBO over the edge variables. On the three degenerate families it is exact. Beyond them, it operates as a generate--prune--verify pipeline in which a Stage~I QUBO generates $k$-regular candidates under a symmetry-breaking prescription, a Stage~II QUBO postselects them through a Max-Clique instance guided by the Hoffman bound, and a final classical $O(N^3)$ test certifies the survivors against Eq.~\eqref{eq:srg-cert}. Its output is in general a reduced set of candidates rather than the solution itself, and the number of candidates to be certified grows with $N$; this is the structural price of avoiding auxiliary variables. The second algorithm, developed in Sec.~\ref{sec:full-qubo}, solves the full problem by quadratizing the quartic penalty through Rosenberg auxiliary variables \cite{rosenberg1975,boros_hammer2002,anthony2017,dattani2019}, so that the global minimizers of a single QUBO are in one-to-one correspondence with the adjacency matrices of the target SRGs. The price is the number of auxiliary variables, which grows as $\Theta(N^3)$ and eventually becomes the limiting factor for sufficiently large graphs. Each route thus pushes the intrinsic hardness of the problem into a different resource, postprocessing volume in the first case and variable count in the second, and their practical domains of applicability turn out to be complementary.

The two routes plug SRG construction into the QUBO solver ecosystem with concrete long-range targets. For Conway's problem the pipeline route requires $\binom{99}{2}=4851$ edge variables, while the complete formulation requires roughly $4.7\times 10^{5}$ binary variables after quadratization, beyond the connectivity-limited capacity of present-day quantum annealers yet within the operating range already demonstrated by state-of-the-art classical QUBO software \cite{pawlowski2025veloxq}. SRG construction thus offers a graded family of benchmark instances, from the Petersen graph to Conway's problem, on which classical, quantum-inspired and quantum QUBO solvers can be systematically compared.

The remainder of the paper is organized as follows. Section~\ref{sec:noqubo} formalizes the obstruction and states the impossibility theorem. Section~\ref{sec:pipeline} develops the pure-QUBO algorithm, exact on the degenerate families and of generate--prune--verify type beyond them, including symmetry breaking, Max-Clique postselection and algebraic certification. Section~\ref{sec:full-qubo} presents the auxiliary-variable algorithm, its variable counting and its empirical convergence. Section~\ref{sec:discussion} closes with a discussion and conclusions. Appendix~\ref{app:no-aux} contains the proofs behind Sec.~\ref{sec:noqubo} together with an exhaustive study of symmetry-broken frames, Appendix~\ref{app:srg-reduction} derives the canonical labelings that reduce the number of free binary parameters in both algorithms, and Appendix~\ref{app:feasibility} reviews the algebraic feasibility conditions on SRG parameters and tabulates the feasible cases for $N\le 49$.

\section{The SRG problem is not an auxiliary free QUBO}
\label{sec:noqubo}
This section establishes the structural fact on which the whole paper is organized. We first show that the defining equation leads to a quartic energy, and then prove that no quadratic energy whatsoever can replace it, apart from three degenerate families. The reader willing to accept these statements at face value may skim the present section and keep only Theorem~\ref{thm:no-aux} and Corollary~\ref{cor:no-aux-class}; complete proofs and extensions are collected in Appendix~\ref{app:no-aux}.

\subsection{From the defining equation to a quartic energy}
\label{subsec:quartic}

We encode an undirected simple graph $G=(V,E)$ with $N$ vertices through binary variables $x_{i,j}$ defined for $0 \le i < j \le N-1$, where $x_{i,j}=1$ if vertices $i$ and $j$ are connected and $0$ otherwise, with the symmetric convention $x_{a,b}\equiv x_{b,a}$ for $a\neq b$. The $k$-regularity condition requires
\begin{equation}\label{eq:degree-constraint}
\sum_{\substack{j=0 \\ j \neq i}}^{N-1} x_{i,j} = k, \quad \forall i \in \{0,\dots,N-1\},
\end{equation}
and admits an exactly quadratic penalty,
\begin{equation}\label{eq:Hgen}
H_{\text{gen}} = \sum_{i=0}^{N-1} \left( \sum_{\substack{j=0 \\ j \neq i}}^{N-1} x_{i,j} - k \right)^2,
\end{equation}
whose global minimizers are precisely the $k$-regular graphs. A positive penalty multiplier could be attached to $H_{\text{gen}}$, but it plays no role in what follows and we set it to unity throughout.

The remaining SRG conditions count common neighbors. For any two distinct vertices $i$ and $j$, the number of common neighbors on a candidate adjacency is
\begin{equation}\label{eq:common-neighbors}
c_{i,j} \;=\; \sum_{\substack{\ell=0 \\ \ell \neq i, j}}^{N-1} x_{i,\ell}\, x_{j,\ell},
\end{equation}
which is quadratic in the edge variables. The SRG conditions demand $c_{i,j}=\lambda$ whenever $i$ and $j$ are adjacent and $c_{i,j}=\mu$ otherwise, which can be written compactly as the linear relation
\begin{equation}\label{eq:srg-condition}
c_{i,j} \;=\; \mu \;+\; (\lambda - \mu)\, x_{i,j}, \qquad 0 \le i < j \le N-1,
\end{equation}
the entrywise form of Eq.~\eqref{eq:srg-cert}. Enforcing Eq.~\eqref{eq:srg-condition} through a standard squared penalty introduces the term
\begin{equation}\label{eq:Hlm}
H_{\lambda\mu} \;=\; \sum_{0\le i < j \le N-1} \Big(\, c_{i,j} \;-\; \mu \;-\; (\lambda-\mu)\, x_{i,j} \,\Big)^{\!2}.
\end{equation}
Expanding the square produces monomials of the form $x_{i,\ell}x_{j,\ell}x_{i,\ell'}x_{j,\ell'}$, that is, products of up to four distinct binary variables. The energy $H_{\lambda\mu}$ is therefore a pseudo-Boolean function of degree four and, together with $H_{\text{gen}}$, defines a HUBO rather than a QUBO problem \cite{lucas2014,boros_hammer2002,dattani2019}.

\subsection{No quadratic energy can replace it}
\label{subsec:noquad}

The quarticity of $H_{\lambda\mu}$ does not by itself settle the matter, because a QUBO formulation of a problem need not follow the defining constraints literally. What a QUBO solver returns is a global minimizer of its cost function, so all that is required is some quadratic function whose global minimizers coincide with the desired solutions. We now show that such a function does not exist beyond a short list of degenerate families.

For a real function $f$ on $\{0,1\}^m$, with $m=\binom{N}{2}$ the number of edge variables, we denote by
\begin{equation}\label{eq:minimizers}
\mathcal{M}(f) \;=\; \Bigl\{ x^\star \in \{0,1\}^m : f(x^\star) \le f(x) \ \ \forall\, x \in \{0,1\}^m \Bigr\}
\end{equation}
the set of its global minimizers, comprising all configurations that attain the minimum value of $f$ over the binary cube. When $f$ is a QUBO cost function, $\mathcal{M}(f)$ is the set of optimal solutions of the instance, the manifold of minimum-energy configurations of the associated classical spin Hamiltonian. Let $S \subset \{0,1\}^m$ collect the labeled adjacency matrices of the SRGs with parameters $(N,k,\lambda,\mu)$, comprising all vertex labelings of all isomorphism classes, and let $R \supseteq S$ collect all labeled $k$-regular graphs on $N$ vertices. We wonder whether some polynomial $f$ of degree at most two in the edge variables, possibly unrelated to the defining equation, satisfies $\mathcal{M}(f)=S$. Demanding the full set $S$, rather than a nonempty subset of it, is the meaningful requirement. A construction based on the parameters alone has no means to single out particular solutions, and the apparent shortcut of encoding one known SRG through its Hamming-distance function presupposes the very solution one is after; Appendix~\ref{app:no-aux} discusses this point in detail.

The obstruction comes from symmetry. Relabeling the vertices by a permutation $\pi\in\Sigma_N$ permutes the edge variables through $x_{i,j}\mapsto x_{\pi(i),\pi(j)}$, and this action maps $S$ onto itself and likewise its complement, since a relabeled SRG is an SRG with the same parameters and a relabeled non-SRG is a non-SRG. The first step is to identify the quadratic functions that are themselves invariant under all relabelings, and it turns out that the space of such functions is remarkably small. The following lemma is proved in Appendix~\ref{app:no-aux}; the underlying principle, that label-invariant polynomials in edge variables are spanned by subgraph counts, is standard in the theory of graph parameters \cite{lovasz2012}.

\begin{lemma}\label{lem:invariant-quadratics}
Every $\Sigma_N$-invariant polynomial of degree at most two in the edge variables is a linear combination, with constant coefficients, of the four functions
\begin{equation}\label{eq:invariant-basis}
1, \quad \mathcal{E}=\!\!\sum_{0\le i<j\le N-1}\!\!x_{i,j}, \quad \mathcal{W}=\sum_{v}\binom{d_v}{2}, \quad \mathcal{D}=\binom{\mathcal{E}}{2}-\mathcal{W},
\end{equation}
where $d_v=\sum_{u\ne v}x_{u,v}$ is the degree of vertex $v$. Here $\mathcal{E}$ counts edges, $\mathcal{W}$ counts pairs of edges sharing a vertex, and $\mathcal{D}$ counts disjoint pairs of edges. Consequently, the value of any such polynomial at a configuration $x$ depends on $x$ only through its degree sequence $\bigl(d_0(x),\dots,d_{N-1}(x)\bigr)$.
\end{lemma}

The physical content of the lemma is that a relabeling-invariant two-body Hamiltonian on edge variables can measure nothing about a graph beyond its degree statistics. In particular it assigns one and the same energy to every $k$-regular graph, and cannot tell the Petersen graph from any other cubic graph on ten vertices. This is all that is needed for the main result.

\begin{theorem}\label{thm:no-aux}
Let $S\neq\emptyset$. There exists a polynomial $f$ of degree at most two in the edge variables, not assumed invariant, with $\mathcal{M}(f)=S$ if and only if $S=R$, that is, if and only if every $k$-regular graph on $N$ vertices is an SRG with parameters $(N,k,\lambda,\mu)$. When $S=R$, the regularity penalty $H_{\mathrm{gen}}$ of Eq.~\eqref{eq:Hgen} is such an encoding.
\end{theorem}

\begin{proof}
If $S=R$, then $H_{\mathrm{gen}}$, which is quadratic, non-negative, and vanishes exactly on the $k$-regular configurations, satisfies $\mathcal{M}(H_{\mathrm{gen}})=R=S$.

Conversely, suppose some quadratic $f$ satisfies $\mathcal{M}(f)=S$, and let $f_{\min}$ be its minimum value, so that $f=f_{\min}$ on $S$ and $f>f_{\min}$ everywhere else. Consider the average of $f$ over all relabelings,
\begin{equation}\label{eq:fbar}
\bar f(x) \;=\; \frac{1}{N!}\sum_{\pi\in\Sigma_N} f(\pi\cdot x),
\end{equation}
which is again a polynomial of degree at most two, and invariant by construction. The averaging does not disturb the set of global minimizers. If $x\in S$, every relabeling $\pi\cdot x$ is again in $S$, so every term in the average equals $f_{\min}$ and $\bar f(x)=f_{\min}$; if $x\notin S$, every relabeling stays outside $S$, so every term is strictly larger than $f_{\min}$ and $\bar f(x)>f_{\min}$. Hence $\mathcal{M}(\bar f)=S$ as well. By Lemma~\ref{lem:invariant-quadratics} there exist constants $\kappa_0,\dots,\kappa_3$ such that $\bar f = \kappa_0 + \kappa_1 \mathcal{E} + \kappa_2 \mathcal{W} + \kappa_3 \mathcal{D}$, and every $k$-regular configuration yields the same three values $\mathcal{E}=Nk/2$, $\mathcal{W}=N\binom{k}{2}$ and $\mathcal{D}=\binom{Nk/2}{2}-N\binom{k}{2}$. Therefore $\bar f$ takes one and the same value on all of $R$. Since $S$ is a nonempty subset of $R$ on which $\bar f=f_{\min}$, that common value must be $f_{\min}$. But then every $k$-regular configuration attains the minimum, so $R\subseteq\mathcal{M}(\bar f)=S$, and therefore $R=S$.
\end{proof}

One might worry that the average should somehow remember $\lambda$ and $\mu$, since $f$ is designed with a target family in mind. It does, but only through its coefficients. Lemma~\ref{lem:invariant-quadratics} states that, whatever the coefficients, the value of an invariant quadratic at a configuration depends on the configuration only through its degree sequence, and degrees carry no information about common neighbors; this is precisely the information that an invariant quadratic is structurally unable to encode, however its coefficients are tuned. Appendix~\ref{app:no-aux} elaborates on this point.

\begin{corollary}\label{cor:no-aux-class}
Among feasible parameter sets, $S=R$ holds exactly in three degenerate situations, namely the perfect matchings $tK_2=\mathrm{SRG}(2t,1,0,0)$, for which every $1$-regular graph is a perfect matching; their complements, the cocktail-party graphs $K_{n\times2}=\mathrm{SRG}(2n,2n-2,2n-4,2n-2)$, since complementation $x\mapsto\mathbf{1}-x$ maps quadratics to quadratics; and the pentagon $C_5=\mathrm{SRG}(5,2,0,1)$, which is the unique $2$-regular graph on five vertices. The empty and the complete graph also satisfy $S=R$, trivially, and are discarded here as degenerate. For every parameter set that admits at least one $k$-regular non-SRG graph on $N$ vertices, no quadratic function of the edge variables has $S$ as its exact set of global minimizers, and auxiliary variables are unavoidable for an exact single-QUBO formulation.
\end{corollary}

The hypothesis of the corollary is satisfied in every non-degenerate case of interest, and in particular for Conway's problem. The disjoint union of the complete graph $K_{15}$, which is $14$-regular, with any $14$-regular graph on $84$ vertices is a $14$-regular graph on $99$ vertices whose adjacent pairs do not share exactly one common neighbor, so $14$-regular non-SRG graphs on $99$ vertices exist. If, on the other hand, no Conway graph exists at all, then $S$ is empty while $\mathcal{M}(f)$ never is. Whether or not Conway's graph exists, therefore, no QUBO over the $4851$ edge variables can have the Conway configurations as its exact set of global minimizers. To our knowledge, this is the first formal no-go statement for auxiliary-free QUBO approaches to this problem.

The three degenerate families of Corollary~\ref{cor:no-aux-class} will reappear in the sampling experiments of Sec.~\ref{sec:pipeline} as a regime of guaranteed success, since for them $\mathcal{M}(H_{\mathrm{gen}})=S$ already holds and $k$-regular generation alone constitutes the whole problem; their role in those experiments is that of a control confirming the boundary of the corollary, not that of an open exploration. Theorem~\ref{thm:no-aux} organizes the remainder of the paper. Section~\ref{sec:pipeline} pushes the pure-QUBO route as far as it can legitimately go, solving the degenerate families exactly and treating all others through a pipeline whose final certification is classical, while Sec.~\ref{sec:full-qubo} pays the price identified by the theorem and solves the full problem through auxiliary variables.

\section{Pure QUBO: exceptional families and a generate--prune--verify pipeline}
\label{sec:pipeline}

\subsection{Scope of the pure-QUBO route}
\label{subsec:scope}

On the three degenerate families of Corollary~\ref{cor:no-aux-class}, minimizing $H_{\mathrm{gen}}$ is already an exact SRG algorithm, since every $k$-regular graph is the target. We verified this boundary by exhaustive enumeration. For $\mathrm{SRG}(5,2,0,1)$ the set $\mathcal{M}(H_{\mathrm{gen}})$ contains $12$ configurations, all of them labeled pentagons, and for the cocktail-party case $\mathrm{SRG}(6,4,2,4)$ it contains $15$ configurations, all of them labeled copies of $K_{3\times2}$; handing $H_{\mathrm{gen}}$ to any QUBO solver therefore solves these families outright. The boundary is sharp. For the parameters $(6,3,0,3)$ the same set contains $70$ configurations, of which only $10$ are labeled copies of $K_{3,3}$, the unique SRG of that family, and Theorem~\ref{thm:no-aux} guarantees that no quadratic repair of $H_{\mathrm{gen}}$ exists.

Outside the degenerate regime, an exact pure-QUBO formulation is forbidden, but the pure-QUBO route remains useful provided one renounces exactness. We adopt a strategy that approaches the problem as closely as a pure QUBO over the edge variables permits, a transparent generate--prune--verify pipeline consisting of two QUBO stages,
\begin{enumerate}
    \item \textit{Stage I: Graph Generation.} Generation of highly restricted $k$-regular candidate graphs through the quadratic penalty $H_{\mathrm{gen}}$ coupled with a symmetry-breaking mechanism.
    \item \textit{Stage II: QUBO-Based Spectral Postselection.} QUBO postselection of the generated candidates using maximum clique analysis and the Hoffman bound.
\end{enumerate}
A final exact algebraic verification against Eq.~\eqref{eq:srg-cert} is then applied to the reduced output of the two QUBO stages. In this way the nontrivial search and postselection remain fully QUBO-based, while the terminal certification is deferred to an exact deterministic check on the surviving candidates. The output of the two QUBO stages is thus a reduced set of solution candidates rather than the solution itself; this is the structural price of avoiding auxiliary variables, and the classical postprocessing volume it entails is the quantity that ultimately limits the reach of the divided strategy as $N$ grows.

\subsection{Stage I: graph generation}

\subsubsection{QUBO matrix form}

The generation Hamiltonian $H_{\mathrm{gen}}$ of Eq.~\eqref{eq:Hgen} is a quadratic polynomial over binary variables and can be expressed in standard QUBO matrix form,
\begin{equation}\label{eq:qubo}
H_{\text{gen}} = \vec{x}^{\,T} Q \vec{x} + \text{const},
\end{equation}
where the additive constant $N k^2$ does not affect minimization and is omitted hereafter.
Expanding the square in Eq.~\eqref{eq:Hgen} reveals the linear terms $Q_{(i,j),(i,j)} = 2(1 - 2k)$ and the quadratic couplings $Q_{(i,j),(i,\ell)} = 2$ whenever two edges share a common vertex; all remaining entries of $Q_{(i,j),(p,q)}$ vanish, up to the symmetric placement induced by the chosen vectorization of $\vec{x}$.

\subsubsection{\texorpdfstring{Symmetry breaking: the $k$-regular fixed-row strategy}{Symmetry breaking: the k-regular fixed-row strategy}}
\label{subsec:fixed-row}

For large values of $N$, the search space of size $2^{N(N-1)/2}$ exhibits a high degree of degeneracy due to the $N!$ possible relabelings of the vertices \cite{babai2015graph}, the same symmetry exploited by the proof of Theorem~\ref{thm:no-aux}, which here becomes a practical obstacle by creating vast families of energetically equivalent configurations. To mitigate this, we introduce an explicit symmetry-breaking mechanism that exploits the homogeneity of the $k$-regular degree sequence.

By anchoring the adjacency pattern of the initial vertices, we force the solver to search within a specific canonical branch of the isomorphism tree, effectively collapsing a large portion of the variables into constants. Assuming $N \ge 2k$, the procedure is systematically executed as follows.
\begin{enumerate}
    \item \textit{Root Vertex Fixing.} Without loss of generality, we anchor vertex $0$ by enforcing its required $k$ connections sequentially. We set $x_{0,j} = 1$ for $j \in \{1,\dots,k\}$, and $x_{0,j} = 0$ for $j \in \{k+1,\dots,N-1\}$.
    \item \textit{Second Vertex Fixing.} Vertex $1$ is already adjacent to vertex $0$ by construction, since $x_{0,1}=1$. To satisfy $k$-regularity, it requires exactly $k-1$ additional connections. To break the remaining symmetry and prevent forced localized substructures, such as early triangle formation, we assign its remaining neighbors to the next available block of vertices. We set $x_{1,j} = 1$ for $j \in \{k+1,\dots,2k-1\}$, and $x_{1,j} = 0$ for all other remaining indices.
\end{enumerate}

Strictly speaking, only the Root Vertex Fixing is without loss of generality for the universe of $k$-regular graphs, since any such graph admits a labeling in which vertex $0$ is connected to the $k$ smallest-indexed vertices. The Second Vertex Fixing further imposes that vertex $1$ has no neighbors among $\{2,\dots,k\}$, which forbids triangles containing the edge $\{0,1\}$ and is therefore a labeling \emph{convention} that restricts the search to the still large subfamily of $k$-regular graphs admitting it. For triangle-free targets with $\lambda=0$, such as the Petersen and Clebsch graphs, the prescription is structurally consistent with the target and coincides with the natural canonical fixing of Sec.~\ref{app:lambda-zero-case}. For general-$\lambda$ targets, however, the prescription is not merely restrictive but exclusive. Vertices $0$ and $1$ are adjacent by construction, their common neighbors must lie in $\{2,\dots,k\}$ by the Root Vertex Fixing, and the Second Vertex Fixing together with $k$-regularity forces $x_{1,j}=0$ on that whole range, so any $k$-regular completion satisfies $c_{0,1}=0$; a target with $\lambda>0$ therefore admits no valid completion at all within this frame. The throughput survey of Fig.~\ref{fig:pipeline} deliberately retains the fixed-row prescription across all families, which turns the general-$\lambda$ targets with $N\ge 2k$ into controlled exposures of this exclusion. In the SRG-aware experiments, namely the benchmark of Sec.~\ref{subsubsec:sa-pipeline} and the auxiliary-variable study of Sec.~\ref{sec:full-qubo-exp}, we use instead the more permissive general-$\lambda$ canonical fixing of Sec.~\ref{app:general-case}, which leaves the $k-1$ entries of row $1$ in the slot $\{2,\dots,k\}$ as free variables subject to the global sum constraint $\sum_{j=2}^k x_{1,j}=\lambda$ and is consistent with any target. The fixed-row prescription of the present subsection should be understood as the $\lambda=0$ specialization of the SRG-aware canonical fixing, exposed here as a standalone symmetry-breaking heuristic for generic $k$-regular graph generation; the benchmarks of Fig.~\ref{fig:benchmark} are run in this $\lambda=0$ specialization on a $k$-regular testbed and illustrate its impact on solver performance.

Injecting these fixed values into $H_{\text{gen}}$ algebraically reduces the dimensionality of the QUBO matrix and redefines the linear biases of the remaining free variables. Figure~\ref{fig:benchmark} illustrates the impact of this symmetry-breaking strategy.

\begin{figure}
    \includegraphics[width=0.9\linewidth]{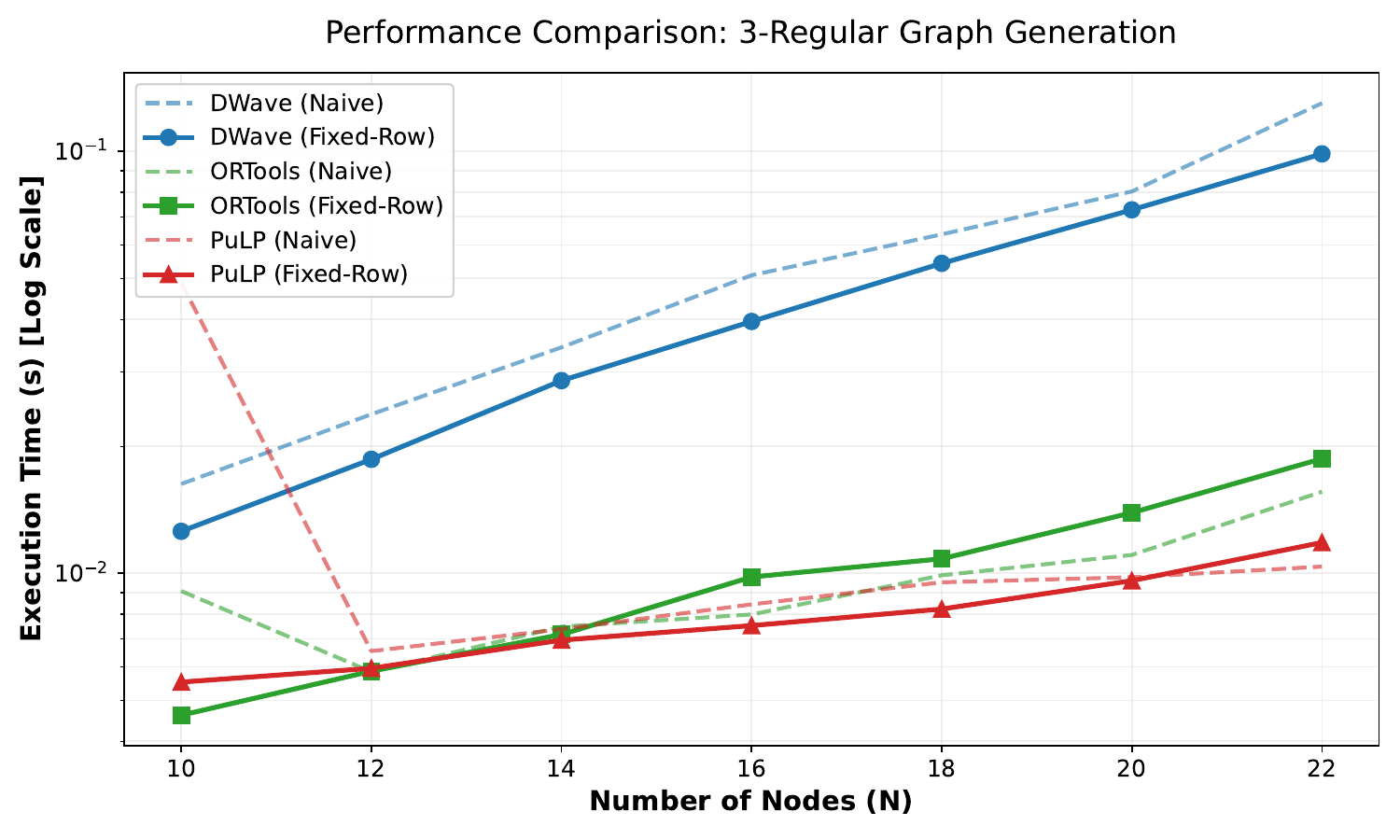}
    \caption{\justifying Effect of the fixed-row symmetry breaking on Stage~I QUBO solution times, evaluated on $k$-regular graph generation with $k=3$ as a controlled testbed; this case isolates the optimization cost of $k$-regularity and is not itself an SRG target. For each solver considered, the solid line shows the wall-clock time of the fixed-row formulation of Sec.~\ref{subsec:fixed-row} and the dashed line shows the naive formulation; the gap between solid and dashed within each solver quantifies the speedup obtained from the symmetry reduction. The figure is intended as a per-solver before and after comparison. Cross-solver readings reflect heterogeneous implementations and access modalities and are not meant as a quantum-versus-classical benchmark. For the D-Wave data in particular, the reported values are end-to-end platform-level elapsed times obtained through remote access and therefore do not resolve device-level quantities such as physical annealing time, number of reads, embedding overhead, or run-by-run success rates; we do not currently have direct access to the hardware to report those metrics separately.}
    \label{fig:benchmark}
\end{figure}

\begin{algorithm}
\SetAlgoLined
\DontPrintSemicolon
\KwIn{Number of nodes $n$, desired regular degree $k$}
\KwOut{Regular graph of degree $k$}
\BlankLine
\textbf{Apply fixed-row symmetry breaking} following Sec.~\ref{subsec:fixed-row}:\;
\quad For $j\in\{1,\dots,k\}$, fix $x_{0,j}\leftarrow 1$; for $j\in\{k+1,\dots,n-1\}$, fix $x_{0,j}\leftarrow 0$.\;
\quad For $j\in\{k+1,\dots,2k-1\}$, fix $x_{1,j}\leftarrow 1$; for all other $j\ge 2$, fix $x_{1,j}\leftarrow 0$.\;
\BlankLine
\textbf{Define} the remaining binary variables $x_{i,j}$ for $i<j$, where $x_{i,j} = 1$ if there is an edge between nodes $i$ and $j$, $0$ otherwise.\;
\BlankLine
\textbf{Construct} the QUBO matrix $Q$ on the unfixed variables; the fixed-row substitution renormalizes the linear biases of the remaining $x_{i,j}$ and removes their contribution as constants from the Hamiltonian:\;
\For{$i \leftarrow 0$ \KwTo $n-1$}{
    \For{$j \leftarrow i+1$ \KwTo $n-1$}{
        \textbf{Initialize} $Q_{(i,j),(i,j)} \leftarrow 0$\;
    }
}
\BlankLine
\For{$i \leftarrow 0$ \KwTo $n-1$}{
    \textbf{Define} $M$ as a large penalty constant.\;
    \textbf{Obtain} the list of unfixed variables associated with node $i$:\;
    $\quad$ $V_i = \{ x_{i,j} \mid j = i+1, \dots, n-1 \} \cup \{ x_{j,i} \mid j = 0, \dots, i-1 \}$\;
    \BlankLine
    \textbf{Add linear terms} to the QUBO:\;
    \ForEach{$x_{i,j} \in V_i$}{
        $Q_{(i,j),(i,j)} \leftarrow Q_{(i,j),(i,j)} + (1 - 2k) \times M$\;
    }
    \BlankLine
    \textbf{Add quadratic terms} to the QUBO:\;
    \ForEach{$x_{i,j}, x_{i,l} \in V_i$ with $x_{i,j} \neq x_{i,l}$}{
        $Q_{(i,j),(i,l)} \leftarrow Q_{(i,j),(i,l)} + 2 \times M$\;
    }
}
\BlankLine
\textbf{Minimize} the objective function:\;
\[
H_{\mathrm{gen}} = \sum_{(i,j)} Q_{(i,j),(i,j)} x_{i,j} + \sum_{(i,j) \neq (p,q)} Q_{(i,j),(p,q)} x_{i,j} x_{p,q}
\]
using a classical or quantum QUBO solver.\;
\BlankLine
\textbf{Reconstruct} the graph $G$ from the fixed-row values and the optimal solution:\;
\For{$i \leftarrow 0$ \KwTo $n-1$}{
    \For{$j \leftarrow i+1$ \KwTo $n-1$}{
        \If{$x_{i,j} = 1$}{
            Add edge between nodes $i$ and $j$ in $G$\;
        }
    }
}
\BlankLine
\textbf{Verify} that each node in $G$ has degree $k$.\;
\If{$G$ is a regular graph of degree $k$}{
    \Return{$G$}
}
\Else{
    \Return{No valid graph found}
}
\caption{\justifying Generation of a Regular Graph via QUBO with Fixed-Row Symmetry Breaking}
\end{algorithm}

\subsection{Stage II: QUBO-based spectral postselection}

In this stage, the input graphs are precisely the candidate graphs produced by Stage I; no externally constructed benchmark graphs are supplied to the postselection routine in deployment. Throughout this subsection, however, the Max-Clique QUBO subroutine is first validated against ground-truth SRG instances of known clique number to ensure that the postselection acts on the actual $\omega(G)$ rather than on a heuristic under-estimate; only after this validation does the subroutine consume Stage I output in the full pipeline. The purpose of this stage is to retain the postselection inside the QUBO setting. For this reason, rather than replacing it by a direct spectral diagonalization of each candidate adjacency matrix, we formulate the Maximum Clique problem as a secondary QUBO and use the Hoffman bound as a necessary filter for SRG structure.

An SRG with parameters $(N,k,\lambda,\mu)$ has exactly three distinct eigenvalues, namely $k$ with multiplicity $1$ and two additional eigenvalues $\theta_2$ and $\theta_3$ uniquely determined by
\begin{equation}\label{eq:lambda23}
    \theta_2, \theta_3 = \frac{(\lambda-\mu) \pm \sqrt{(\lambda-\mu)^2 + 4(k-\mu)}}{2}.
\end{equation}
This rigid spectral structure determines the Hoffman bound expected for the target SRG family and sharply constrains global combinatorial properties, such as the clique number $\omega(G)$, which represents the size of the largest fully connected subgraph. The Hoffman bound \cite{hoffman_1970,delsarte_1973_thesis,haemers2021} establishes that any $k$-regular graph with smallest eigenvalue $\theta_{\min}$ satisfies
\begin{equation}
    \omega(G) \le 1 + \frac{k}{|\theta_{\min}|}.
\end{equation}
For an SRG with parameters $(N,k,\lambda,\mu)$ the relevant smallest eigenvalue is $\theta_{\min} = \theta_3$, the lower root of Eq.~\eqref{eq:lambda23}. Since $\omega(G)$ must be an integer, the bound naturally implies $\omega(G) \le \lfloor 1 + k/|\theta_{\min}| \rfloor$. For many SRG families, this bound is tight. We use the bound here as a filter on Stage~I candidates by computing $|\theta_{\min}|$ from the \emph{target} parameters $(\lambda,\mu,k)$, not from the candidate's own spectrum. A candidate graph $G_{\mathrm{cand}}$ that is in fact isomorphic to the target SRG must satisfy $\omega(G_{\mathrm{cand}}) \le \lfloor h \rfloor$, where $h = 1 + k/|\theta_{\min}^{\mathrm{target}}|$ is the Hoffman value of the target family, so failure of the inequality is sufficient to reject $G_{\mathrm{cand}}$ as a candidate for that specific target. The converse is not true, since a $k$-regular graph that is \emph{not} the target SRG may still satisfy the inequality; the bound is only a necessary condition for target-equality, and not in general a necessary condition for being some other SRG, let alone for being non-SRG. The filter is therefore properly read as a consistency test against the spectrum of the target family, rather than as a generic spectral test on $G_{\mathrm{cand}}$; the rejection of false-positive non-SRGs is the responsibility of the final algebraic certification of Sec.~\ref{subsubsec:cert}. By formulating the Maximum Clique problem as a secondary QUBO, we can test the target-equality necessary condition while keeping the postselection inside the QUBO framework.

\subsubsection{QUBO formulation for clique verification}

Given the adjacency matrix of a candidate graph $G_{\mathrm{cand}}$, we introduce binary variables $z_i \in \{0,1\}$, where $z_i = 1$ denotes that vertex $i$ belongs to the clique. The corresponding QUBO reads
\begin{equation}\label{eq:qubo_clique_final}
    H_{\mathrm{clique}} = -\gamma \sum_{i=0}^{N-1} z_i + \delta \sum_{(u,v) \notin E(G_{\mathrm{cand}})} z_u z_v,
\end{equation}
where $\delta > \gamma > 0$. The first term incentivizes maximizing the clique size, while the second penalizes the inclusion of non-adjacent pairs. To guarantee valid cliques, we strictly set $\delta > \gamma$; otherwise, the solver could incorrectly lower the energy by including a disconnected vertex, as the linear reward $\gamma$ would outweigh the quadratic penalty $\delta$. 

Minimizing $H_{\mathrm{clique}}$ yields the observed clique number $\omega_{\mathrm{obs}} = \sum_{i=0}^{N-1} z_i^\star$. If $\omega_{\mathrm{obs}}$ respects the theoretical Hoffman bound for the target SRG, the candidate is retained by the QUBO postselection stage and passed to the final exact algebraic verification. We have validated this premise empirically on five representative SRG families, as reported in Table~\ref{tab:validation}. For each of them the Max-Clique QUBO recovers the true clique number $\omega_{\mathrm{obs}}=\omega$ in twenty out of twenty independent simulated-annealing runs at $n_{\mathrm{reads}}=100$, confirming that the postselection acts on the actual $\omega(G)$ rather than on a heuristic under-estimate. The same table also reports a non-SRG control, the $5$-prism graph $C_5\square K_2$, also known as the pentagonal prism, which satisfies $\omega_{\mathrm{obs}}\le \lfloor h \rfloor$ with the Petersen Hoffman bound but is rejected by the algebraic verification of Sec.~\ref{subsubsec:cert}, illustrating the necessity of that final step.

\begin{algorithm}
\SetAlgoLined
\DontPrintSemicolon
\KwIn{Candidate $k$-regular graph $G_{\mathrm{cand}}=(V,E)$ on $N$ vertices, target SRG parameters $(\lambda,\mu)$, penalty weights $\gamma>0$ and $\delta>\gamma$}
\KwOut{Decision: \textsc{Retain} or \textsc{Reject}}
\BlankLine
\textbf{Compute} the Hoffman bound for the target SRG family:\;
$\quad D \leftarrow (\lambda-\mu)^2 + 4(k-\mu)$\;
$\quad \theta_{\min} \leftarrow \big((\lambda-\mu) - \sqrt{D}\big)/2$\;
$\quad h \leftarrow 1 + k/|\theta_{\min}|$\;
\BlankLine
\textbf{Define} binary variables $z_i\in\{0,1\}$ for $i=0,\dots,N-1$, where $z_i=1$ if vertex $i$ is selected into the clique.\;
\BlankLine
\textbf{Construct} the Max-Clique QUBO over $G_{\mathrm{cand}}$:\;
\[
H_{\mathrm{clique}} = -\gamma\sum_{i=0}^{N-1} z_i + \delta \!\!\!\!\sum_{(u,v)\notin E} \!\!\!\! z_u\, z_v
\]
\BlankLine
\textbf{Minimize} $H_{\mathrm{clique}}$ using a classical or quantum QUBO solver, yielding $\{z_i^\star\}_{i=0}^{N-1}$.\;
\BlankLine
\textbf{Compute} the observed clique number:\;
$\quad \omega_{\mathrm{obs}} \leftarrow \sum_{i=0}^{N-1} z_i^\star$\;
\BlankLine
\If{$\omega_{\mathrm{obs}} \le \lfloor h \rfloor$}{
    \Return{\textsc{Retain}}\;
}
\Else{
    \Return{\textsc{Reject}}\;
}
\caption{QUBO-Based Spectral Postselection via the Hoffman Bound}
\end{algorithm}

\begin{table*}[tbp]
\centering
\caption{\justifying Empirical validation of the Stage~II QUBO-based postselection. For each target, we report the integer-floored Hoffman bound $\lfloor h \rfloor$, the theoretical clique number $\omega$ from algebraic graph theory, and the observed clique number $\omega_{\mathrm{obs}}$ obtained by minimizing the Max-Clique QUBO of Eq.~\eqref{eq:qubo_clique_final} with $\gamma=1$, $\delta=2$ on \texttt{neal} simulated annealing at $n_{\mathrm{reads}}=100$ and $n_{\mathrm{reps}}=20$ independent runs per target. In all five SRG targets the QUBO recovers the true clique number $\omega_{\mathrm{obs}}=\omega$ in twenty out of twenty repetitions, with Wilson $95\%$ lower confidence bound at least $0.84$, validating the premise that the postselection acts on the actual $\omega(G)$ rather than on a heuristic under-estimate. The last row is a non-SRG control. The $5$-prism graph $C_5\square K_2$ is $3$-regular with $\omega=2$ and therefore $\omega_{\mathrm{obs}} \le \lfloor h \rfloor$ with the Petersen-graph Hoffman bound, demonstrating that the Hoffman filter is necessary but not sufficient; this candidate is rejected by the final exact algebraic verification of Sec.~\ref{subsubsec:cert}, since its non-adjacent pairs share $0$, $1$, or $2$ common neighbors and the constant-$\mu$ condition fails.}
\label{tab:validation}
\begin{tabular}{lccccc}
\hline
Target & $N$ & $\lfloor h \rfloor$ & $\omega$ & $\omega_{\mathrm{obs}}$ & Relation \\
\hline
Rook $L_2(4)$        & 16 & 4 & 4 & 4 & $\omega_{\mathrm{obs}}=\omega=\lfloor h \rfloor$ \\
Triangular $T_6$     & 15 & 5 & 5 & 5 & $\omega_{\mathrm{obs}}=\omega=\lfloor h \rfloor$ \\
Petersen             & 10 & 2 & 2 & 2 & $\omega_{\mathrm{obs}}=\omega=\lfloor h \rfloor$ \\
Paley $P_{13}$       & 13 & 3 & 3 & 3 & $\omega_{\mathrm{obs}}=\omega=\lfloor h \rfloor$ \\
Paley $P_{17}$       & 17 & 4 & 3 & 3 & $\omega_{\mathrm{obs}}=\omega<\lfloor h \rfloor$ \\
\hline
$5$-prism graph $C_5\square K_2$\textsuperscript{$\dagger$} & 10 & 2 & 2 & 2 & passes Hoffman; fails certification \\
\hline
\end{tabular}
\\[3pt]
{\footnotesize $^\dagger$\,Non-SRG control, $3$-regular and triangle-free, with the same $\omega$ as the Petersen graph.}
\end{table*}

\subsubsection{Algebraic certification}
\label{subsubsec:cert}

The Max-Clique QUBO provides a necessary QUBO-based postselection criterion, but it is not sufficient for SRG certification. The surviving candidates are therefore subjected to a final exact algebraic verification of the SRG constraints. Let $A$ be the adjacency matrix of a candidate graph that has passed Stages I and II. We certify the instance by checking the defining equation~\eqref{eq:srg-cert}, together with the conditions $A_{ii}=0$ and $A\mathbf{1}=k\mathbf{1}$. Equivalently, one may verify that every adjacent pair shares exactly $\lambda$ common neighbors and every non-adjacent pair shares exactly $\mu$ common neighbors.

The role of this final step is different from that of the two QUBO stages. Its cost is polynomial in $N$ for a given candidate adjacency matrix, for example $O(N^3)$ using dense matrix operations or an equivalent pairwise common-neighbor count, whereas the hard part of the problem lies in the combinatorial search over the $2^{N(N-1)/2}$ admissible adjacency assignments addressed by the QUBO stages. The exact verification therefore closes the pipeline with a deterministic certification step on a drastically reduced candidate set, rather than replacing the original search problem.

To quantify the discriminative power of the QUBO postselection beyond a single example, we have run exhaustive enumeration of the Stage~I universe for every SRG target in the Brouwer table of Appendix~\ref{app:feasibility} with $N \le 10$, namely the pentagon $C_5=\mathrm{SRG}(5,2,0,1)$, the Paley graph $\mathrm{SRG}(9,4,1,2)$, and the Petersen graph $\mathrm{SRG}(10,3,0,1)$. For each target we enumerate all non-isomorphic connected $k$-regular graphs on $N$ vertices via the standard generator \texttt{geng} from the \texttt{nauty} package \cite{mckay_piperno_nauty}, compute their clique numbers exactly via the Bron--Kerbosch algorithm, retain those satisfying the Hoffman acceptance condition $\omega \le \lfloor h \rfloor$, and finally test the algebraic SRG identity~\eqref{eq:srg-cert} on each survivor. The results are summarized in Table~\ref{tab:hoffman-fp}.

\begin{table}[tbp]
\centering
\caption{\justifying Exhaustive enumeration of the Stage~I universe for every SRG target in the Brouwer table of Appendix~\ref{app:feasibility} with $N\le 10$. For each, we list the integer-floored Hoffman bound $\lfloor h \rfloor$, the size of the Stage~I universe of connected $k$-regular non-isomorphic graphs on $N$ vertices, enumerated via \texttt{nauty-geng}~\cite{mckay_piperno_nauty}, the number of survivors of the Hoffman filter, namely those with $\omega \le \lfloor h \rfloor$, the number of those survivors that further satisfy the algebraic SRG identity, and the resulting false-positive rate of the filter on its own survivors. The pentagon is a degenerate case in which the universe contains a single connected $2$-regular graph, the cycle $C_5$ itself, which is also the SRG, in agreement with Corollary~\ref{cor:no-aux-class}. For the two non-trivial cases, the false-positive rate is strictly above $80\%$, demonstrating that the Hoffman filter is a necessary but far from sufficient condition for SRG structure and that the final algebraic verification of Sec.~\ref{subsubsec:cert} is essential rather than auxiliary.}
\label{tab:hoffman-fp}
\begin{tabular}{lcccccc}
\hline
Target & $(N,k,\lambda,\mu)$ & $\lfloor h \rfloor$ & Universe & Survivors & SRGs & FP rate \\
\hline
Pentagon $C_5$ & $(5,2,0,1)$  & $2$ & $1$  & $1$  & $1$ & $0.0\%$ \\
Paley $P_9$    & $(9,4,1,2)$  & $3$ & $16$ & $14$ & $1$ & $92.9\%$ \\
Petersen       & $(10,3,0,1)$ & $2$ & $19$ & $6$  & $1$ & $83.3\%$ \\
\hline
\end{tabular}
\end{table}

For $N\ge 13$ the universe of connected $k$-regular graphs on $N$ vertices exceeds $10^6$ instances and exhaustive enumeration with off-the-shelf tools is no longer feasible within a reasonable computational budget. On the small-$N$ frontier, however, the false-positive rate is consistently large. The Hoffman filter prunes a substantial fraction of the universe yet retains many non-SRG candidates, confirming quantitatively the role of the algebraic verification of Sec.~\ref{subsubsec:cert} as the final certification step of the pipeline.

To complement this exhaustive small-$N$ enumeration with an end-to-end characterization that extends to parameter regimes inaccessible to direct enumeration, we have conducted a throughput survey across the $18$ feasible SRG targets with $N \le 18$. For each family, $10^5$ independent simulated-annealing reads are drawn on the Stage~I Hamiltonian $H_{\mathrm{gen}}$ under the family-agnostic fixed-row prescription of Sec.~\ref{subsec:fixed-row}, applied in full when $N \ge 2k$ and reduced to the root-vertex anchoring otherwise; every read is decoded and tested for $k$-regularity, the integer-floored Hoffman bound of the target family, and the defining equation~\eqref{eq:srg-cert}. The Stage~I yield is essentially saturated throughout, with at least $99.6\%$ of reads returning a $k$-regular graph, so the survey directly probes the two filters that follow. The results, summarized in Fig.~\ref{fig:pipeline}, organize the $18$ targets into four regimes. A \emph{trivial} regime comprises the pentagon $C_5$ and the cocktail-party graphs $K_{n\times 2}$ with $\mu = k$, seven families in all; for these, Corollary~\ref{cor:no-aux-class} predicts the outcome in advance, since every $k$-regular candidate is the target SRG and both filters are vacuous, and the observed essentially perfect yield is therefore not a finding but an experimental control, confirming the boundary of the corollary and validating the implementation. An \emph{excluded} regime comprises $L_2(3)$, Paley$(13)$, $\mathrm{GQ}(2,2)$, $L_2(4)$ and Paley$(17)$, the five general-$\lambda$ targets to which the full fixed-row prescription applies; their yield vanishes for the structural reason made precise in Sec.~\ref{subsec:fixed-row}, namely that the $\lambda=0$ anchor excludes every valid completion of a target with $\lambda>0$, so these zeros measure the restrictiveness of the fixing rather than the difficulty of the search. A \emph{partial} regime, formed by the Petersen graph and the triangular graph $T_5$, yields a small but strictly positive number of SRG candidates, namely $376$ verified Petersen labelings, $24$ of them distinct, and $300$ verified $T_5$ labelings, $214$ of them distinct, per $10^5$ reads. Finally, an \emph{unreached} regime, comprising the Clebsch graph, the triangular graph $T_6$, and the complements of the Clebsch graph and of $L_2(4)$, yields zero SRGs for genuinely empirical reasons, even though the Hoffman filter often retains almost the entire Stage~I universe, above $99.7\%$ for $T_6$ and the complement of $L_2(4)$. The excluded and unreached regimes together quantify the practical severity of the fundamental obstruction established in Sec.~\ref{sec:noqubo} and provide the direct experimental motivation for the auxiliary-variable formulation developed in Sec.~\ref{sec:full-qubo}, in which the $(\lambda,\mu)$ conditions are enforced at the level of the Hamiltonian rather than relegated to a downstream filter; the deeper structural reason is that the Hoffman bound, by construction, cannot encode the constant-$\mu$ requirement, so enforcing it at the Hamiltonian level is the natural next step. Section~\ref{subsubsec:sa-pipeline} shows, in addition, that the excluded families become reachable once the fixing is upgraded to the SRG-aware canonical prescription of Appendix~\ref{app:srg-reduction}.

\begin{figure}
    \centering
    \includegraphics[width=\linewidth]{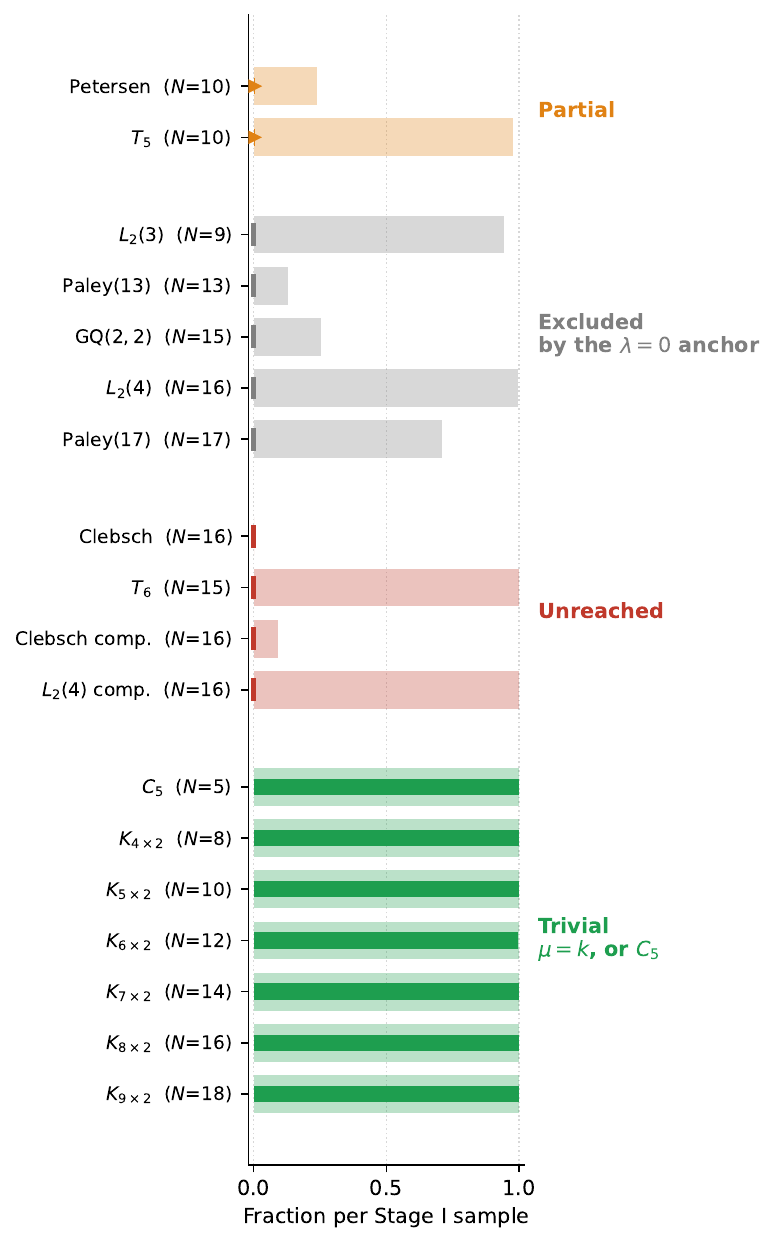}
    \caption{\justifying Throughput survey of the divided pipeline across the 18 feasible SRG families with $N \le 18$, at $10^5$ simulated-annealing reads per family on the Stage~I Hamiltonian $H_{\mathrm{gen}}$ with the fixed-row prescription of Sec.~\ref{subsec:fixed-row}. For each family, the wide, lightly shaded outer bar is the fraction of Stage~I candidates that pass the Hoffman filter, and the narrow, saturated inner bar is the end-to-end SRG yield per read, both drawn in the color of the corresponding regime; a vertical tick at zero marks a vanishing yield, and a small arrowhead marks a yield too small to be visible on this scale. The four regimes, trivial, excluded, partial and unreached, are characterized in the main text.}
    \label{fig:pipeline}
\end{figure}

\subsubsection{SA-based pipeline throughput and structure of false positives}
\label{subsubsec:sa-pipeline}

The survey of Fig.~\ref{fig:pipeline} characterizes the pipeline under the family-agnostic fixed-row prescription. To measure what the pipeline achieves once the fixing is made SRG-aware, we report the end-to-end throughput under simulated annealing on the Stage~I Hamiltonian $H_{\mathrm{gen}}$, now with the general-$\lambda$ canonical fixing of Sec.~\ref{app:general-case} applied, for six representative targets that span the regimes identified by Fig.~\ref{fig:pipeline} and overlap with the four targets characterized for the auxiliary-variable formulation in Sec.~\ref{sec:full-qubo-exp}. The first three rows of Table~\ref{tab:sa-pipeline} cover the partial and excluded regimes of Fig.~\ref{fig:pipeline} at small $N$, namely the Petersen graph, the triangular graph $T_5$, and $L_2(3)$, the latter isomorphic to the Paley graph on nine vertices. The last three rows are common with the auxiliary-QUBO benchmark, namely Paley$(13)$, the triangular graph $T_6$, and the Clebsch graph, enabling a direct empirical comparison between the two formulations on the same targets. For each target the SA returns $n_{\mathrm{reads}}=10^4$ candidate configurations on $H_{\mathrm{gen}}$; each returned sample is reconstructed into an adjacency matrix and tested in turn for $k$-regularity, the integer-floored Hoffman bound of the target family, and the algebraic SRG identity~\eqref{eq:srg-cert}.

\begin{table*}[tbp]
\centering
\caption{\justifying SA-based throughput of the divided pipeline of Sec.~\ref{sec:pipeline} at $n_{\mathrm{reads}}=10^4$ on the Stage~I Hamiltonian $H_{\mathrm{gen}}$ with the general-$\lambda$ canonical fixing of Sec.~\ref{app:general-case}. The column Stage~I counts samples that decode to a $k$-regular graph; Stage~II counts the further samples whose clique number respects the integer-floored Hoffman bound $\lfloor h \rfloor$ of the target family; SRGs verified counts the samples that additionally satisfy the algebraic SRG identity~\eqref{eq:srg-cert}; unique counts the distinct adjacency matrices at the SRG-verified stage and is omitted when only labelings of a single orbit are recovered. The first three rows are small-$N$ controls, with Petersen and $T_5$ in the partial regime of Fig.~\ref{fig:pipeline} and $L_2(3)$ in the excluded regime, where the fixed-row anchor rules out any valid completion. The last three rows are common with the auxiliary-variable benchmark of Sec.~\ref{sec:full-qubo-exp}, reported in Table~\ref{tab:srg-aux-results}, and enable a direct cross-formulation comparison.}
\label{tab:sa-pipeline}
\begin{tabular}{lcccccc}
\hline
Target & $(N,k,\lambda,\mu)$ & $\lfloor h \rfloor$ & Stage~I & Stage~II & SRGs verified & unique \\
\hline
Petersen        & $(10,3,0,1)$ & $2$ & $9998$  & $2290$ & $78$  & $21$ \\
Triangular $T_5$& $(10,6,3,4)$ & $4$ & $9999$  & $9784$ & $58$  & $12$ \\
$L_2(3)$        & $( 9,4,1,2)$ & $3$ & $10000$ & $9759$ & $169$ & $ 4$ \\
\hline
Paley$(13)$     & $(13,6,2,3)$ & $3$ & $10000$ & $1123$ & $0$ & --- \\
Triangular $T_6$& $(15,8,4,4)$ & $5$ & $10000$ & $9977$ & $0$ & --- \\
Clebsch         & $(16,5,0,2)$ & $2$ & $10000$ & $0$    & $0$ & --- \\
\hline
\end{tabular}
\end{table*}

Three observations emerge from the small-$N$ rows of Table~\ref{tab:sa-pipeline}. First, the Stage~I yield under SA is essentially saturated for all targets, with at least $99.98\%$ of reads returning a $k$-regular graph, matching the saturation already observed in the survey of Fig.~\ref{fig:pipeline}; the combinatorially significant work of generating $k$-regular candidates is done by SA on $H_{\mathrm{gen}}$. Second, the conditional Stage~II yield depends sharply on the target. For Petersen, where $\lfloor h \rfloor = 2$ and most non-Petersen $3$-regular graphs on $10$ vertices contain at least one triangle and hence have clique number at least three, the Hoffman filter retains only $22.9\%$ of Stage~I survivors; for $T_5$ and $L_2(3)$, where $\lfloor h \rfloor$ is $4$ and $3$ respectively and clique numbers of arbitrary $k$-regular graphs on the same vertex set typically lie close to or below this bound, the Hoffman filter retains essentially the entire Stage~I universe, above $97\%$ in both cases. The discriminative power of the Hoffman filter is therefore highly $\lambda$-dependent in this regime, and is strongest for triangle-free targets. Third, and most importantly for the empirical positioning of the divided pipeline, the SA-based pipeline reaches the SRG ground state for all three small-$N$ targets within the budget considered, including $L_2(3)$, which is structurally excluded in the survey of Fig.~\ref{fig:pipeline} by the fixed-row anchor. The verified SRG yield is small in absolute terms, at most $1.7\%$ of Stage~I survivors, but strictly positive, and the diversity of recovered configurations, with $21$ distinct labeled adjacency matrices for Petersen, $12$ for $T_5$, and $4$ for $L_2(3)$, confirms that SA does not collapse onto a single labeling but explores the residual orbit of the canonical fixing. Reachability in the divided pipeline is therefore governed by the choice of fixing rather than by the sampler; restoring the general-$\lambda$ canonical fixing turns a family that the fixed-row anchor excludes outright into one solved $169$ times in $10^4$ reads.

The last three rows of Table~\ref{tab:sa-pipeline}, which overlap with the auxiliary-variable benchmark of Sec.~\ref{sec:full-qubo-exp}, expose three structurally distinct failure modes of the divided pipeline beyond the small-$N$ regime. First, for Paley$(13)$ the Hoffman filter retains a modest $11.2\%$ of Stage~I survivors, but none of those survivors satisfy the algebraic SRG identity; the false-positive content of the Hoffman filter saturates the Stage~II yield, and the divided pipeline returns no SRG candidate within the trial budget. Second, for the triangular graph $T_6$ the Hoffman bound $\lfloor h \rfloor = 5$ is consistent with the clique number of essentially any $8$-regular graph on $15$ vertices, so $99.8\%$ of Stage~I survivors pass Stage~II and the filter is effectively trivial; the algebraic certification then rejects all of them. Third, for the Clebsch graph the situation is the opposite extreme. The Hoffman bound $\lfloor h \rfloor = 2$ is so tight that triangle-freeness is required, and the fraction of triangle-free $5$-regular graphs on $16$ vertices among the SA samples is essentially zero, so the Hoffman filter rejects the entire Stage~I universe and no candidate survives to Stage~II. These three failure modes, a strong Hoffman filter saturated by false positives, a trivial Hoffman filter, and a prohibitively tight Hoffman filter, delineate the regime in which the divided pipeline becomes ineffective and motivate the complete formulation of Sec.~\ref{sec:full-qubo} as a complementary route. In the auxiliary-variable benchmark of Table~\ref{tab:srg-aux-results}, Clebsch is recovered with $p_{\mathrm{succ}}\simeq 0.95$ at $n_{\mathrm{reads}}=10^5$ and Paley$(13)$ with $p_{\mathrm{succ}}=0.65$ at $n_{\mathrm{reads}}=2\times 10^5$, precisely on the targets where the divided pipeline fails here. The two formulations are therefore empirically complementary. The divided pipeline is competitive on $\lambda=0$ targets whose Hoffman bound is informative without being prohibitive, of which Petersen is the prototype, while the auxiliary-variable formulation extends to targets whose constraint structure cannot be captured by a spectral filter alone, of which Paley and Clebsch are prototypes. Triangular $T_6$ is the only target on which neither formulation succeeds within the budgets considered here, consistent with the structural argument given in Sec.~\ref{sec:discussion}.

\paragraph*{Structure of the false positives.}
For each target we examined the Hoffman-passing non-SRG candidates returned by Stage~II, sampling fifty false positives per target, and computed their common-neighbor statistics. For the Petersen target, all false positives are $3$-regular triangle-free graphs, so the $\lambda=0$ condition is automatically satisfied and $\operatorname{var}(c_{i,j}\,|\,\text{adjacent})=0$, but the common-neighbor count on non-adjacent pairs deviates from the target $\mu=1$; pooled across the fifty false positives, $60.8\%$ of non-adjacent pairs have $c_{i,j}=1$, while $19.6\%$ have $c_{i,j}=0$ and $19.6\%$ have $c_{i,j}=2$. The $5$-prism graph used in Table~\ref{tab:validation} as a sharp non-SRG control is therefore not an isolated example; the entire false-positive pool for the Petersen target consists of triangle-free $3$-regular graphs that fail only the constant-$\mu$ condition, with $c_{i,j}$ varying around the target value. For the $T_5$ and $L_2(3)$ targets the failure modes are more diverse, since the adjacent and non-adjacent common-neighbor counts both vary, with $\operatorname{var}(c_{i,j}\,|\,\text{adjacent})\simeq 0.4$ and $\operatorname{var}(c_{i,j}\,|\,\text{non-adjacent})$ between $0.1$ and $0.4$, so the false positives simultaneously violate both the $\lambda$- and $\mu$-conditions. For Paley$(13)$ and $T_6$ the false-positive pool exhibits the broadest spread of $c_{i,j}$ values; the non-adjacent common-neighbor counts range over five or six distinct integer values around the target $\mu$, with the modal bin at $\mu$ accounting for less than half of the pairs, namely $46.3\%$ for Paley$(13)$ and $37.9\%$ for $T_6$. The Hoffman filter therefore acts as a strong filter for $\lambda=0$ targets, where it tightly bounds the clique number and forces triangle-freeness, and as a weaker filter for $\lambda>0$ targets, where the clique-number bound is consistent with a broad range of common-neighbor structures. In both regimes, however, the final algebraic certification of Sec.~\ref{subsubsec:cert} is essential.

\section{Solving the full problem with auxiliary variables}
\label{sec:full-qubo}

Theorem~\ref{thm:no-aux} leaves a single route towards an exact formulation, namely enlarging the variable space. The empirical findings of Sec.~\ref{sec:pipeline} point in the same direction. As shown in Fig.~\ref{fig:pipeline}, the divided strategy is fully effective only over a restricted region of parameter space. For $9$ of the $18$ algebraically feasible targets with $N \le 18$ that we surveyed, the pipeline returns no SRG candidate within $10^5$ Stage~I reads, five of those failures being structural exclusions imposed by the fixed-row anchor and four genuine search failures; and of the nine families where it does succeed, seven are the degenerate ones whose success is guaranteed by Corollary~\ref{cor:no-aux-class}, so that the survey genuinely conquers only the Petersen graph and the triangular graph $T_5$. In this section we show that, at the price of a controlled polynomial blow-up in the number of binary variables, the entire SRG problem can be cast as a single QUBO instance, with no further heuristic postselection beyond a polynomial-time algebraic check on the QUBO output. This formulation is valuable for several reasons. First, a complete QUBO encoding provides a self-contained benchmark target for classical and quantum-inspired QUBO solvers \cite{lucas2014,glover2018,pawlowski2025veloxq} as well as for quantum annealers \cite{kadowaki1998}. Second, it exposes the algebraic structure of the SRG constraints in a form that is directly amenable to Ising-model implementations and therefore fits naturally within the same theoretical framework used in earlier QUBO-based studies of combinatorial graph problems on quantum hardware \cite{bian2014,chapuis2019,Jain_2021,Grant_Humble_Stump_2021}. Third, it offers a concrete instance of how the standard technique of \emph{quadratization} of higher-order pseudo-Boolean functions \cite{rosenberg1975,boros_hammer2002,ishikawa2011,anthony2017,dattani2019} can be deployed in a combinatorial graph-theoretic setting of independent interest. The auxiliary variables are not a circumstantial artifact of our construction; by Theorem~\ref{thm:no-aux} they cannot be avoided by any exact quadratic formulation, however ingenious.

\subsection{Quadratization via Rosenberg auxiliary variables}

The systematic reduction of a pseudo-Boolean function of degree larger than two to a quadratic one by introducing auxiliary binary variables is known as quadratization \cite{rosenberg1975,boros_hammer2002,ishikawa2011,anthony2017,dattani2019}. Our starting point is the quartic penalty $H_{\lambda\mu}$ of Eq.~\eqref{eq:Hlm}, whose quarticity originates in the products $x_{i,\ell}\,x_{j,\ell}$ appearing squared inside the common-neighbor counts of Eq.~\eqref{eq:common-neighbors}. The strategy is to give each such product a name. For each ordered triple $(i,j,\ell)$ with $0 \le i<j \le N-1$ and $\ell \notin \{i,j\}$, we introduce an auxiliary binary variable $y_{i,j,\ell}$ constrained to represent the product
\begin{equation}\label{eq:aux-def}
y_{i,j,\ell} \;=\; x_{i,\ell}\, x_{j,\ell}.
\end{equation}
The identity \eqref{eq:aux-def} can be enforced at the level of the Hamiltonian through Rosenberg's quadratization penalty \cite{rosenberg1975,boros_hammer2002},
\begin{equation}\label{eq:rosenberg}
P(x_1, x_2, y) \;=\; 3\, y \;+\; x_1 x_2 \;-\; 2\, x_1 y \;-\; 2\, x_2 y,
\end{equation}
which is manifestly quadratic in $(x_1,x_2,y)$ and has the following property. A direct inspection of the eight truth-table assignments of $(x_1,x_2,y)\in\{0,1\}^3$ shows that $P=0$ if and only if $y=x_1 x_2$, and $P\ge 1$ otherwise. Applying \eqref{eq:rosenberg} to every triple $(i,j,\ell)$ yields the auxiliary penalty
\begin{equation}\label{eq:Haux}
H_{\text{aux}} \;=\; \sum_{\substack{0\le i<j\le N-1 \\ \ell \neq i,j}} P\big(x_{i,\ell},\, x_{j,\ell},\, y_{i,j,\ell}\big),
\end{equation}
which is quadratic in the extended variable set $\{x_{i,j}\} \cup \{y_{i,j,\ell}\}$. Under this change of variables, the common-neighbor count \eqref{eq:common-neighbors} becomes linear,
\begin{equation}\label{eq:cij-aux}
c_{i,j} \;=\; \sum_{\substack{\ell=0 \\ \ell \neq i,j}}^{N-1} y_{i,j,\ell},
\end{equation}
and the SRG penalty \eqref{eq:Hlm} reduces to the quadratic expression
\begin{equation}\label{eq:Hlm-aux}
H_{\lambda\mu}^{\mathrm{(q)}} \;=\; \sum_{0\le i < j \le N-1} \!\!\Bigg(\, \sum_{\substack{\ell=0 \\ \ell \neq i,j}}^{N-1} y_{i,j,\ell} \;-\; \mu \;-\; (\lambda-\mu)\, x_{i,j} \,\Bigg)^{\!2}.
\end{equation}
Equations~\eqref{eq:Haux} and~\eqref{eq:Hlm-aux} together eliminate all terms of degree higher than two from the original HUBO Hamiltonian.

\subsection{The full SRG Hamiltonian}

Combining the $k$-regularity penalty $H_{\text{gen}}$ of Eq.~\eqref{eq:Hgen}, the common-neighbor penalty $H_{\lambda\mu}^{\mathrm{(q)}}$, and the quadratization penalty $H_{\text{aux}}$, the complete QUBO Hamiltonian for the SRG problem with parameters $(N,k,\lambda,\mu)$ reads
\begin{equation}\label{eq:H-SRG}
H_{\mathrm{SRG}} \;=\; w_1\, H_{\text{gen}} \;+\; w_2\, H_{\lambda\mu}^{\mathrm{(q)}} \;+\; w_3\, H_{\text{aux}},
\end{equation}
with positive multipliers $w_1, w_2, w_3 > 0$. These weights have to be chosen so that no violation of the underlying constraints can lower the total energy below a configuration that satisfies all of them simultaneously. Following the standard prescription for constrained pseudo-Boolean problems \cite{lucas2014,glover2018}, $w_3$ must dominate the couplings between each $y_{i,j,\ell}$ and the rest of the Hamiltonian, so that every auxiliary variable locks onto its defining product~\eqref{eq:aux-def} at the ground state; $w_1$ and $w_2$ are then set large enough that the degree and common-neighbor constraints cannot be simultaneously violated at lower cost than satisfied.

A worst-case lower bound on $w_3$ that guarantees auxiliary lock at every ground state can be derived explicitly. Inspection of the Rosenberg polynomial \eqref{eq:rosenberg} shows that every unlocked triple $(x_{i,\ell},x_{j,\ell},y_{i,j,\ell})$ contributes at least $w_3$ to the energy. Flipping a single $y_{i,j,\ell}$ shifts the common-neighbor count $c_{i,j}$ of Eq.~\eqref{eq:cij-aux} by $\pm 1$, so the maximum reduction in the squared penalty $w_2 (c_{i,j}-\mu-(\lambda-\mu)x_{i,j})^2$ achievable by a single flip is bounded by $w_2 (2\Delta-1)$, where $\Delta$ is the largest possible signed deviation of $c_{i,j}$ from its target. Since $0 \le c_{i,j} \le N-2$, one has $\Delta \le N-2$ and the sufficient condition
\begin{equation}\label{eq:A3-bound}
w_3 \;\ge\; w_2\,(2N-5)
\end{equation}
guarantees that no single-flip move reduces the total energy by unlocking an auxiliary.

A complementary, much tighter bound is obtained by restricting attention to configurations that are at most one step away from the SRG ground state, which are precisely the configurations SA visits with high probability in the late stages of annealing. At the SRG ground state itself, all common-neighbor constraints are saturated, with $\Delta = 0$ for every pair, so unlocking a single auxiliary necessarily \emph{increases} $H_{\lambda\mu}^{\mathrm{(q)}}$ by $w_2$ and incurs the Rosenberg penalty $w_3 \ge 1$; the lock is automatic for any $w_3 > 0$. At a single-step neighbor of the ground state, with one constraint violated by $\Delta = 1$, the maximum savings from unlocking an auxiliary to restore the constraint is $w_2(2\Delta-1) = w_2$, so the sufficient condition for the lock to dominate becomes simply
\begin{equation}\label{eq:A3-bound-tight}
w_3 \;\gtrsim\; w_2.
\end{equation}
Empirically, as reported in Sec.~\ref{sec:full-qubo-exp}, this near-ground-state bound is consistent with the optimum $w_3 = w_2 = 1$ identified on the experimental grid. At this choice the lock is essentially enforced, with a mean unlocked count of at most $7.2$ out of $758$ auxiliaries, roughly one percent, for the targets surveyed, while the energy landscape remains permissive enough for SA to navigate. The worst-case bound \eqref{eq:A3-bound} is therefore overly conservative for the typical SA trajectory toward the SRG ground state, and the experimental recipe $w_1 = w_2 = w_3 = 1$ should be read as the saturation of the tighter near-ground-state estimate \eqref{eq:A3-bound-tight}.

By construction, the global minimizers of $H_{\mathrm{SRG}}$, projected onto the original edge variables $\{x_{i,j}\}$, are in one-to-one correspondence with the adjacency matrices of strongly regular graphs with parameters $(N,k,\lambda,\mu)$. The whole problem therefore reduces to a single QUBO minimization whose ground states encode the target SRGs; a trivial $O(N^3)$ algebraic verification, step~9 of Algorithm~\ref{alg:srg-aux}, is sufficient to certify that a returned configuration is a genuine ground state rather than a low-lying excited state. A substantial reduction in the number of free binary variables can be obtained by anchoring three rows of the adjacency matrix according to the canonical fixing detailed in Appendix~\ref{app:srg-reduction}, an SRG-aware extension of the fixed-row strategy of Sec.~\ref{subsec:fixed-row} that exploits the parameters $(\lambda,\mu)$ to fix additional entries of $A$ beyond rows~$0$ and~$1$. The implementation of this canonical fixing in the auxiliary-variable QUBO, including the propagation of fixed entries to the auxiliary variables $\{y_{i,j,\ell}\}$ and the resulting variable count, is described in Appendix~\ref{app:srg-reduction-qubo}.

\begin{algorithm}
\SetAlgoLined
\DontPrintSemicolon
\KwIn{SRG parameters $(N,k,\lambda,\mu)$, penalty weights $w_1, w_2, w_3 > 0$}
\KwOut{Adjacency matrix of an SRG with parameters $(N,k,\lambda,\mu)$, or fail}
\BlankLine
\textbf{Apply canonical fixing} following Appendix~\ref{app:srg-reduction}:\;
\quad Fix entries of rows $0$, $1$ and $k+1$ of $A$ via Eqs.~(\ref{eq:row0})--(\ref{eq:rowk1-general}) for general $\lambda$, or Eqs.~(\ref{eq:row0-l0})--(\ref{eq:rowk1-l0}) when $\lambda=0$.\;
\quad \textbf{Propagate} the fixings to the auxiliary variables: for every triple $(i,j,\ell)$ such that both $x_{i,\ell}$ and $x_{j,\ell}$ are fixed, set $y_{i,j,\ell} \leftarrow x_{i,\ell}\, x_{j,\ell}$.\;
\BlankLine
\textbf{Define} the remaining edge variables $\{x_{i,j}\}$ and auxiliary variables $\{y_{i,j,\ell}\}$ as binary unknowns.\;
\BlankLine
\textbf{Construct} the SRG Hamiltonian:\;
\[
H_{\mathrm{SRG}} \;=\; w_1\, H_{\mathrm{gen}} \;+\; w_2\, H^{(q)}_{\lambda\mu} \;+\; w_3\, H_{\mathrm{aux}}
\]
\quad with $H_{\mathrm{gen}}$, $H^{(q)}_{\lambda\mu}$ and $H_{\mathrm{aux}}$ given by Eqs.~\eqref{eq:Hgen}, \eqref{eq:Hlm-aux} and \eqref{eq:Haux}.\;
\BlankLine
\textbf{Minimize} $H_{\mathrm{SRG}}$ using a classical or quantum QUBO solver, yielding $\{x_{i,j}^\star\}$ and $\{y_{i,j,\ell}^\star\}$.\;
\BlankLine
\textbf{Reconstruct} the adjacency matrix $A$ from the canonical fixings and the optimized free variables $\{x_{i,j}^\star\}$.\;
\BlankLine
\textbf{Verify} the SRG conditions: $A\mathbf{1}=k\mathbf{1}$ and $A^2 = (\lambda-\mu)A + (k-\mu)\mat{I} + \mu\mat{J}$.\;
\BlankLine
\If{both conditions hold}{
    \Return{$A$}\;
}
\Else{
    \Return{No valid SRG found}\;
}
\caption{SRG Construction via Auxiliary-Variable QUBO}
\label{alg:srg-aux}
\end{algorithm}

\subsection{Variable counting and practical scaling}

The benefit of a fully QUBO encoding comes at the cost of a polynomial increase in the number of binary variables. The set of edge variables has cardinality $N(N-1)/2$, whereas the set of auxiliary variables $\{y_{i,j,\ell}\}$ indexed by $0 \le i<j \le N-1$ and $\ell \notin \{i,j\}$ has cardinality
\begin{equation}\label{eq:aux-count}
\binom{N}{2}\,(N-2) \;=\; \frac{N(N-1)(N-2)}{2}.
\end{equation}
The counts in this subsection refer to the QUBO before any canonical fixing is applied; the corresponding counts after the canonical fixing of Appendix~\ref{app:srg-reduction} are reported in Sec.~\ref{app:srg-reduction-qubo}, where the number of remaining edge variables $\varepsilon$ is given by Eqs.~\eqref{eq:F-general}--\eqref{eq:F-lambda0} and the number of remaining auxiliary variables $\alpha$ by Eq.~\eqref{eq:aux-free-count}.
The total number of binary variables therefore grows as $\Theta(N^3)$, in contrast with the $\Theta(N^2)$ scaling of the Stage~I generation QUBO of Sec.~\ref{sec:pipeline}. For the $99$-vertex target of Conway's problem, the count evaluates to roughly $4.7\times 10^{5}$ binary variables before canonical fixing, which quantifies the entrance fee that Theorem~\ref{thm:no-aux} shows to be unavoidable. The quadratic couplings present in $H_{\mathrm{SRG}}$ are sparse but numerous. Each auxiliary variable is coupled to two edge variables through \eqref{eq:Haux}, and each pair $(i,j)$ contributes an $(N-2)\times(N-2)$ block of quadratic couplings among the $y_{i,j,\ell}$ arising from the expansion of \eqref{eq:Hlm-aux}. The resulting interaction graph is therefore denser than that of Stage~I, which directly impacts the minor-embedding overhead on sparsely connected quantum annealing architectures \cite{boothby2016,denchev2016,Grant_Humble_Stump_2021}.

\subsection{Empirical convergence under simulated annealing}
\label{sec:full-qubo-exp}

To probe the practical reach of this formulation in the regime where the divided pipeline of Sec.~\ref{sec:pipeline} is empirically blocked at the Hoffman filter, we have characterized its convergence on four representative SRG targets covering both canonical-fixing branches by simulated annealing, with the canonical fixing of Appendix~\ref{app:srg-reduction} applied to $H_{\mathrm{SRG}}$ in each case. For each target, the ground-state manifold of the canonically fixed Hamiltonian consists of those labeled adjacency matrices of the target SRG, unique up to isomorphism in the four cases considered, that are compatible with the prescribed fixings of rows $0$, $1$ and $k+1$. The residual relabeling freedom is the subgroup of $\mathrm{Aut}(G_{\mathrm{target}})$ that preserves the canonical anchor, namely pointwise fixation of vertex $0$, setwise fixation of its neighborhood $\mathcal{L}_1$ with vertex $1$ also pointwise fixed, and pointwise fixation of vertex $k+1$ in $\mathcal{L}_2$. The success rate $p_{\mathrm{succ}}$ reported below is per-run and counts a run as successful whenever simulated annealing recovers any element of this orbit. Convergence is assessed through the residual violation
\begin{equation}\label{eq:T-min}
\tau_{\min} \;=\; \tau_{\mathrm{reg}} \,+\, \tau_\lambda \,+\, \tau_\mu,
\end{equation}
where $\tau_{\mathrm{reg}} = \sum_i (\deg_i - k)^2$ counts squared $k$-regularity violations and $\tau_\lambda, \tau_\mu$ count squared violations of the $\lambda$- and $\mu$-conditions on the reconstructed adjacency matrix. A converged ground state corresponds to $\tau_{\min} = 0$. We report the empirical success rate $p_{\mathrm{succ}}$, defined as the fraction of independent repetitions reaching $\tau_{\min}=0$, with $95\%$ confidence intervals computed by the Wilson method \cite{wilson1927}. Throughout, point estimates of $p_{\mathrm{succ}}$ should be read together with their Wilson intervals, since the modest sample size $n_{\mathrm{reps}}=20$ adopted here yields intervals of non-negligible width even at the saturated endpoints; for instance, a point estimate of $1.00$ is consistent with a true success rate as low as $0.84$ at the $95\%$ level.

\begin{table}[tbp]
\centering
\caption{\justifying Empirical success rate $p_{\mathrm{succ}}$ of the auxiliary-variable QUBO of Sec.~\ref{sec:full-qubo} on four SRG targets, two in the $\lambda=0$ canonical-fixing branch, Petersen and Clebsch, and two in the general-$\lambda$ branch, Paley$(13)$ and the triangular graph $T_6$, as a function of the simulated-annealing budget $n_{\mathrm{reads}}$, with the canonical fixing of Appendix~\ref{app:srg-reduction} applied and $w_1=w_2=w_3=1$. Each entry is computed over $n_{\mathrm{reps}}=20$ independent repetitions; the $95\%$ confidence interval is the Wilson interval for a binomial proportion. The column of free variables lists the size of the reduced QUBO after canonical fixing has been substituted into Eq.~\eqref{eq:H-SRG}.}
\label{tab:srg-aux-results}
\begin{tabular}{lccccc}
\hline
SRG & $(N,k,\lambda,\mu)$ & free vars & $n_{\mathrm{reads}}$ & $p_{\mathrm{succ}}$ & $95\%$ CI \\
\hline
Petersen     & $(10,3,0,1)$ & $252$  & $10^4$            & $\mathbf{1.00}$ & $[0.84,\,1.00]$ \\
Petersen     & $(10,3,0,1)$ & $252$  & $5\!\times\!10^4$ & $\mathbf{1.00}$ & $[0.84,\,1.00]$ \\
\hline
Paley$(13)$  & $(13,6,2,3)$ & $816$  & $10^4$            & $0.00$ & $[0.00,\,0.16]$ \\
Paley$(13)$  & $(13,6,2,3)$ & $816$  & $3\!\times\!10^4$ & $0.15$ & $[0.05,\,0.36]$ \\
Paley$(13)$  & $(13,6,2,3)$ & $816$  & $5\!\times\!10^4$ & $0.20$ & $[0.08,\,0.42]$ \\
Paley$(13)$  & $(13,6,2,3)$ & $816$  & $10^5$            & $0.45$ & $[0.26,\,0.66]$ \\
Paley$(13)$  & $(13,6,2,3)$ & $816$  & $2\!\times\!10^5$ & $\mathbf{0.65}$ & $[0.43,\,0.82]$ \\
\hline
Triangular $T_6$ & $(15,8,4,4)$ & $1317$ & $10^4$            & $0.00$ & $[0.00,\,0.16]$ \\
Triangular $T_6$ & $(15,8,4,4)$ & $1317$ & $5\!\times\!10^4$ & $0.00$ & $[0.00,\,0.16]$ \\
\hline
Clebsch     & $(16,5,0,2)$ & $1404$ & $10^4$            & $0.10$ & $[0.03,\,0.30]$ \\
Clebsch     & $(16,5,0,2)$ & $1404$ & $3\!\times\!10^4$ & $0.55$ & $[0.34,\,0.74]$ \\
Clebsch     & $(16,5,0,2)$ & $1404$ & $5\!\times\!10^4$ & $0.85$ & $[0.64,\,0.95]$ \\
Clebsch     & $(16,5,0,2)$ & $1404$ & $10^5$            & $0.95$ & $[0.76,\,0.99]$ \\
Clebsch     & $(16,5,0,2)$ & $1404$ & $2\!\times\!10^5$ & $\mathbf{0.95}$ & $[0.76,\,0.99]$ \\
\hline
\end{tabular}
\end{table}

\begin{figure}
    \centering
    \includegraphics[width=0.9\linewidth]{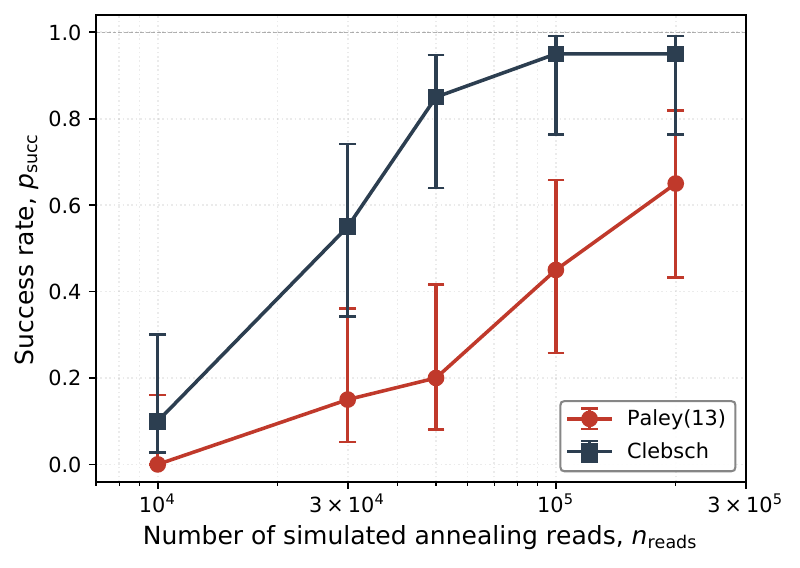}
    \caption{\justifying Empirical success rate $p_{\mathrm{succ}}$ of the auxiliary-variable QUBO of Sec.~\ref{sec:full-qubo} as a function of the simulated-annealing budget $n_{\mathrm{reads}}$, for Paley$(13)$, red circles, and the Clebsch graph, dark squares. Each point is averaged over $20$ independent repetitions; error bars are Wilson $95\%$ confidence intervals for a binomial proportion. Both curves grow monotonically with $n_{\mathrm{reads}}$; Clebsch saturates at $p_{\mathrm{succ}}\simeq 0.95$ for $n_{\mathrm{reads}} \ge 10^5$, while Paley$(13)$ reaches $0.65$ at $n_{\mathrm{reads}}=2\times 10^5$ and has not yet saturated. The behaviour of Petersen and triangular $T_6$, reported in Table~\ref{tab:srg-aux-results}, is not shown here; Petersen saturates at $p_{\mathrm{succ}}=1$ already at $n_{\mathrm{reads}}=10^4$, and $T_6$ produces no success in $40$ runs across both probed budgets.}
    \label{fig:psuccess}
\end{figure}

The empirical success rate of the auxiliary-variable QUBO on the four targets is summarized in Table~\ref{tab:srg-aux-results} and Fig.~\ref{fig:psuccess}. Three qualitative features emerge. First, $p_{\mathrm{succ}}$ increases monotonically with $n_{\mathrm{reads}}$ for Paley$(13)$ and Clebsch, and saturates near unity for Clebsch at $n_{\mathrm{reads}} \ge 10^5$. The convergence behavior of the formulation is therefore governed by the sampling budget rather than by a structural barrier. With sufficiently many reads, simulated annealing recovers the ground state, even though the absolute budget required grows steeply with the size of the reduced QUBO. Second, the four targets cluster cleanly into two regimes set by the canonical-fixing branch. The two $\lambda=0$ targets, Petersen and Clebsch, converge readily, with Petersen reaching $p_{\mathrm{succ}}=1$ already at $n_{\mathrm{reads}}=10^4$ and Clebsch saturating at $p_{\mathrm{succ}}\simeq 0.95$ at $n_{\mathrm{reads}}=10^5$, whereas the two general-$\lambda$ targets behave very differently, since Paley$(13)$ reaches $p_{\mathrm{succ}}=0.65$ only at $n_{\mathrm{reads}}=2\times 10^5$ and triangular $T_6$ fails to produce any ground state in $40$ independent runs at $n_{\mathrm{reads}} \le 5\times 10^4$. Third, the branch separation is not explained by the reduced-QUBO size alone. Clebsch, with $1404$ free variables, converges more readily than Paley$(13)$, with $816$, at every $n_{\mathrm{reads}}$, and Petersen, with $252$, outperforms $T_6$, with $1317$. The refined fixing of Sec.~\ref{app:lambda-zero-case} eliminates strictly more entries of the adjacency matrix than the general prescription of Sec.~\ref{app:general-case} and forces those additional entries to zero specifically, matching the low edge density of the target SRG and producing an energy landscape with a substantially larger basin of attraction around the SRG ground state.

The sensitivity of the formulation to the auxiliary-coupling weight $w_3$, which controls the strength of the Rosenberg quadratization penalty in Eq.~\eqref{eq:Haux}, exhibits a sharp optimum at $w_3=1$ on the one-dimensional sweep with $w_1=w_2=1$ explored here. A sweep over $w_3 \in \{0.5, 1.0, 1.5, 2.0, 5.0\}$ at $n_{\mathrm{reads}}=10^5$, with $20$ independent repetitions per point and per target, yields $p_{\mathrm{succ}} = 0.05,\,0.80,\,0.55,\,0.15,\,0.00$ for Clebsch and $p_{\mathrm{succ}} = 0.05,\,0.45,\,0.15,\,0.00,\,0.00$ for Paley$(13)$. The inverted-U dependence on $w_3$ along this one-dimensional cut is sharply resolved, with a unique optimum at $w_3=1$ for both targets and a monotonically decreasing trend on both sides of it on the grid examined. The mechanism is transparent in the auxiliary-lock diagnostic. At $w_3=0.5$, the mean number of auxiliary variables $y_{i,j,\ell}$ deviating from $x_{i,\ell}x_{j,\ell}$ at the returned ground state is $22.6$ out of $1326$ for Clebsch and $15.6$ out of $758$ for Paley$(13)$, indicating that the quadratization identity is partially broken and the optimization is effectively decoupled from the SRG conditions. At $w_3=1$ the lock is sharply enforced, and the mean unlocked count drops to $1.9$ out of $1326$ for Clebsch and $7.2$ out of $758$ for Paley$(13)$. For $w_3 \ge 1.5$ all auxiliaries lock onto their defining products to within $\sim 0.2$ unlocked on average, but the over-constrained landscape sharply degrades the search, so that the success rate decreases monotonically with $w_3$ from this point onwards and reaches zero by $w_3=5$ on both targets. The value $w_3=1$ is the unique choice in the sweep at which the quadratization penalty is strong enough to enforce auxiliary locking on a substantial fraction of runs while remaining permissive enough for simulated annealing to escape local minima. We do not claim optimality over the full three-parameter penalty landscape or on a finer grid; the empirical optimum $w_3=1$ should be understood relative to the choices $w_1=w_2=1$ and the discrete grid on which it was identified.

\paragraph*{Experimental conditions for Table~\ref{tab:srg-aux-results} and Fig.~\ref{fig:psuccess}.}
The auxiliary-variable QUBO experiments were performed on the same $64$-core, $128$~GB DDR5 workstation, using the \texttt{neal} simulated annealing sampler of the D-Wave Ocean SDK with $w_1=w_2=w_3=1$ unless otherwise stated. Each run is an independent simulated annealing trajectory on $H_{\mathrm{SRG}}$ with the canonical fixing of Appendix~\ref{app:srg-reduction} applied, initialized from the default random state of the \texttt{neal} sampler, at one of the budgets $n_{\mathrm{reads}}\in\{10^4,3\times 10^4,5\times 10^4,10^5,2\times 10^5\}$ for the Paley$(13)$ and Clebsch success-rate curves, at $n_{\mathrm{reads}}\in\{10^4,5\times 10^4\}$ for the Petersen and triangular $T_6$ rows, and at $n_{\mathrm{reads}}=10^5$ for the $w_3$ sweep over $\{0.5,1.0,1.5,2.0,5.0\}$. The stopping criterion is the default cooling schedule of the sampler. Each data point is the empirical success rate over $20$ independent repetitions; reported confidence intervals are Wilson $95\%$ binomial intervals \cite{wilson1927}. A run is counted as a success if and only if $\tau_{\min}=0$ on the adjacency matrix reconstructed from the optimized edge variables $\{x_{i,j}^\star\}$ via the canonical fixing of Appendix~\ref{app:srg-reduction}, which by construction is necessary and sufficient for the reconstructed graph to satisfy the algebraic SRG identity~\eqref{eq:srg-cert}. The criterion is evaluated on the edge variables alone and is therefore independent of whether the auxiliary variables $\{y_{i,j,\ell}^\star\}$ have locked onto their defining products. The residual coincides with the value of $H_{\mathrm{SRG}}$ at the returned configuration up to a constant whenever the auxiliary lock is perfect; partial auxiliary unlocking, observed at $w_3=0.5$ as discussed above, does not invalidate the criterion, since the projected adjacency is tested directly. The full characterization for the four targets of Table~\ref{tab:srg-aux-results} and the $w_3$ sweep comprises $780$ independent runs and completed in approximately $3.5$~hours of wall-clock time at $32$-way parallelism.

We therefore regard the complete QUBO formulation of Eq.~\eqref{eq:H-SRG} as the exact counterpart, mandated by Theorem~\ref{thm:no-aux}, of the heuristic pipeline of Sec.~\ref{sec:pipeline}, rather than as its computational competitor on present-generation devices. The two-stage pipeline uses only $\Theta(N^2)$ binary variables at each QUBO step and exploits the algebraic structure of the Hoffman bound as an efficient necessary filter, while the complete formulation enforces the full $(\lambda,\mu)$ conditions at the level of the Hamiltonian and is better suited as a benchmarking and analytical reference. This situation is reminiscent of other combinatorial problems for which quadratization is formally available but generates a prohibitive variable overhead in practice \cite{bian2014,neukart2017,ishikawa2011,anthony2017}.

\section{Discussion and Conclusions}
\label{sec:discussion}

The central result of this work is a sharp delimitation of what quadratic binary optimization can and cannot do for strongly regular graph construction. Theorem~\ref{thm:no-aux} shows that a quadratic function of the edge variables whose global minimizers are exactly the labeled SRGs exists only when every $k$-regular graph is already the target, which happens for the perfect matchings, the cocktail-party graphs and the pentagon, and for no other feasible parameter set. The result covers arbitrary quadratics, symmetric or not, and therefore closes the door on any hope of an ingenious auxiliary-free encoding. In particular, we have formally established that the $99$-vertex graph of Conway's \$1000 problem cannot be obtained, whether or not it exists, as the exact ground-state set of any QUBO over the $4851$ edge variables; to our knowledge this is the first no-go statement of this kind for a prominent open problem in algebraic graph theory. The proof is a short symmetry argument. Any candidate quadratic can be averaged over the vertex-relabeling group without disturbing its minimizers, and the surviving invariant quadratics can only measure degree statistics, which are blind to the common-neighbor structure that distinguishes an SRG.

The two algorithms developed here follow the geometry drawn by the theorem. The pure-QUBO route of Sec.~\ref{sec:pipeline} is exact precisely on the three degenerate families, where minimizing $H_{\mathrm{gen}}$ solves the problem outright, and beyond them it operates as a generate--prune--verify pipeline. Candidate generation and Hoffman-based postselection remain quadratic, the vertex-relabeling degeneracy is tamed by the fixed-row symmetry breaking of Sec.~\ref{subsec:fixed-row}, and the terminal certification against the defining equation~\eqref{eq:srg-cert} is classical and polynomial. Its output is a reduced candidate set rather than the solution, and the certification volume is the resource that grows with $N$. The auxiliary-variable route of Sec.~\ref{sec:full-qubo} pays instead the entrance fee identified by the theorem. Rosenberg quadratization renders the full problem a single QUBO whose global minimizers are in one-to-one correspondence with the target adjacency matrices, at the cost of $\Theta(N^3)$ auxiliary variables, roughly $4.7\times 10^{5}$ for Conway's problem.

The two routes are not only complementary in scope but empirically complementary on the same targets, as the cross-comparison of Tables~\ref{tab:sa-pipeline} and~\ref{tab:srg-aux-results} makes explicit. The pipeline succeeds on Petersen, with $78$ verified SRGs in $10^4$ SA reads on $H_{\mathrm{gen}}$, and on the small-$N$ controls $T_5$ and $L_2(3)$, but fails on Paley$(13)$, on $T_6$ and on Clebsch within the same budget, for the three structurally distinct reasons identified in Sec.~\ref{subsubsec:sa-pipeline}, namely a Hoffman filter saturated by false positives, a trivial Hoffman filter, and a prohibitively tight Hoffman filter. The auxiliary-variable formulation succeeds on Clebsch, with $p_{\mathrm{succ}}\simeq 0.95$ at $n_{\mathrm{reads}}=10^5$, and on Paley$(13)$, with $p_{\mathrm{succ}}=0.65$ at $n_{\mathrm{reads}}=2\times 10^5$, precisely where the pipeline is blocked, while neither route reaches $T_6$ within the budgets considered. A practitioner aiming at a specific SRG should therefore select the route according to the family-specific discriminativeness of the Hoffman bound, which can be assessed in advance from the parameters alone.

The empirical contrast between the two canonical-fixing branches admits a structural reading. The number of free binary edge parameters after the canonical fixing of Appendix~\ref{app:srg-reduction} is $F=(N^2-5N+2k-2\mu+6)/2$ for general $\lambda$ and $F_0=(N^2-7N+12)/2$ in the triangle-free branch, as established in Observations~\ref{obs:count-general} and~\ref{obs:count-lambda0}. The reduction $F-F_0=N+k-\mu-3$ counts the additional entries eliminated by the $\lambda=0$ prescription, and these are not generic removals but forced zeros that align the canonical adjacency with the sparse background of a low-density SRG, producing a substantially larger basin of attraction around the ground state than the variable count alone would suggest. The ordering observed in Table~\ref{tab:srg-aux-results}, where both $\lambda=0$ targets converge at smaller budgets than both general-$\lambda$ targets even though Clebsch has the largest reduced search space of the four, is consistent with this geometric picture. Within the general-$\lambda$ branch, the failure of $T_6$ relative to Paley$(13)$ is likewise structural. The coincidence $\lambda=\mu=4$ degrades the discriminative slope of the linear relation~\eqref{eq:srg-condition} that separates edges from non-edges, and the higher density of active constraints densifies the interaction graph and flattens the landscape, both effects acting against simulated annealing within a fixed schedule.

Beyond the theorem, the contribution of this work is architectural. We provide the first QUBO treatments of SRG construction, organized as the only two consistent responses to the impossibility result, together with the canonical labelings of Appendix~\ref{app:srg-reduction} that substantially reduce the number of free parameters in both routes, and the feasibility table of Appendix~\ref{app:feasibility}, which lists the $33$ admissible parameter quadruples for $N\le 49$ and serves as the input lookup for instantiating either formulation.

We close with a remark on deployment. Both formulations are solver-agnostic QUBO instances, and their most immediate prospects lie in the rapidly advancing ecosystem of classical and quantum-inspired QUBO technology. Modern high-performance solvers such as VeloxQ \cite{pawlowski2025veloxq} routinely process instances several orders of magnitude larger than the reduced QUBOs characterized in Sec.~\ref{sec:full-qubo-exp}, and tensor-network methods now simulate quantum optimization hardware at scales that were until recently out of reach \cite{kshetrimayum2026tn}; deploying the two formulations on such platforms, with the canonical fixings of Appendix~\ref{app:srg-reduction} applied, is the natural next step and requires no new theoretical ingredient. A concrete long-range target for this program is Conway's $99$-graph problem, recently attacked with SAT solvers \cite{keramatipour2026sat}. Our results frame the QUBO attack precisely, since Theorem~\ref{thm:no-aux} rules out any auxiliary-free shortcut, the pipeline route requires $4851$ edge variables, and the complete formulation requires roughly $4.7\times 10^{5}$ binary variables after quadratization, demanding but within the demonstrated operating range of current classical QUBO software. The SRG family, graded by the feasibility table of Appendix~\ref{app:feasibility}, thereby provides a natural benchmark suite on which classical, quantum-inspired and quantum QUBO solvers can be systematically compared, with Conway's problem as its most prominent open target.

Quantum hardware remains a complementary avenue within this picture. The Stage~I generation QUBO of Sec.~\ref{sec:pipeline} requires $\Theta(N^2)$ logical variables and could in principle be deployed at $N$ ranging into the low hundreds on a current annealing QPU such as D-Wave Advantage2, with Zephyr topology, roughly $7000$ qubits and native connectivity $20$, subject to minor-embedding overhead \cite{boothby2016}, whereas the $\Theta(N^3)$ scaling of the auxiliary-variable formulation confines a fully quantum treatment to the smallest targets. More promising in the near term is a hybrid quantum--classical division of labor, for which the generate--prune--verify pipeline is naturally suited. A quantum or quantum-inspired sampler handles the Stage~I candidate generation, while the Hoffman filter and the algebraic certification run classically, an arrangement that becomes advantageous precisely in the regimes where quantum sampling outperforms classical and tensor-network alternatives \cite{kshetrimayum2026tn}. The structural relevance of such synthesis extends to the construction of quantum-information objects that are combinatorially supported by SRGs, including mutually unbiased bases from Paley graphs \cite{Bandyopadhyay2001} and equiangular line sets associated with prescribed SRG parameters \cite{LEMMENS1973494}, where explicit realizations remain rare.

Open directions include reducing residual symmetries beyond the canonical fixing of three rows of $A$, exploring alternative quadratizations whose energy landscape is more amenable to annealing, and characterizing the dependence of the Hoffman-filter discriminativeness on the SRG parameters in regimes beyond those considered here. A further intriguing possibility, suggested by the structure of Eq.~\eqref{eq:H-SRG}, is an interpolated annealing protocol in which the multipliers $w_2$ and $w_3$ are switched on slowly along the schedule, deforming the landscape continuously from the $k$-regular generation Hamiltonian $H_{\mathrm{gen}}$ to the full SRG Hamiltonian with its Rosenberg auxiliaries; such a protocol admits analysis both as a classical annealing schedule and as a quantum one, where it maps almost natively onto the programmable schedules of annealing hardware. A systematic study of these schedules goes beyond the present scope, but the two formulations developed here provide the fixed endpoints between which any such interpolation would run.

\section*{Acknowledgements}
This work was supported by the FONDECYT Regular Grant Number 1230586, Chile, and by VRI scholarship from Pontificia Universidad Católica de Chile. The authors used Claude (Anthropic) as a writing-assistance tool to improve the clarity and style of the English text. All authors reviewed and edited the resulting text and take full responsibility for the content of this work.

\appendix

\section{Impossibility of auxiliary-free encodings: proofs and exhaustive frame study}
\label{app:no-aux}

This appendix supplies the proofs behind Sec.~\ref{sec:noqubo} and extends the analysis in two directions, first to relabeling-invariant penalties of degree three, and second to the symmetry-broken frames produced by the canonical fixing of Appendix~\ref{app:srg-reduction}, for which we settle the existence question exhaustively on the two smallest instances of this work.

\subsection{Exactness and the trap of a posteriori encodings}
\label{app:no-aux-setup}

We first justify the requirement $\mathcal{M}(f)=S$ adopted in Sec.~\ref{subsec:noquad}. If one merely asks for a quadratic whose global minimizers form a nonempty subset of $S$, the question trivializes. For any single known solution $x^\star\in S$, the Hamming distance $\sum_i (x_i - x^\star_i)^2$ is a quadratic function on the binary cube whose unique global minimizer is $x^\star$, so a subset encoding always exists as soon as an SRG exists. Nothing is gained, because writing this function down requires possessing the solution beforehand. The meaningful question concerns formulations that are constructible from the parameters $(N,k,\lambda,\mu)$ alone. Such a construction has no information with which to prefer one solution over another, and any penalty built from the defining constraints, being a sum of terms that vanish when the constraints hold and are positive otherwise, is automatically constant on $S$ and strictly positive outside it, for every choice of the weights. The condition $\mathcal{M}(f)=S$ is the mathematical expression of this even-handedness, and it is the condition that Theorem~\ref{thm:no-aux} shows to be unattainable beyond the degenerate families.

\subsection{Proof of Lemma~\ref{lem:invariant-quadratics}}
\label{app:no-aux-lemma}

\begin{proof}
A polynomial of degree at most two is a linear combination of the monomials $1$, $x_e$, and $x_e x_{e'}$ with $e\neq e'$, where $e$ and $e'$ run over edge slots, that is, over unordered pairs of distinct vertices. Under a relabeling, monomials map to monomials, so an invariant polynomial must assign the same coefficient to all monomials in the same $\Sigma_N$-orbit and is therefore a combination of the orbit sums. The orbits are easily listed. All single slots form one orbit, whose sum is $\mathcal{E}$. Pairs of distinct slots sharing exactly one vertex form a second orbit, whose sum is $\mathcal{W}$, because choosing two edges at a common vertex $v$ can be done in $\binom{d_v}{2}$ ways. Pairs of disjoint slots form the third orbit, whose sum is the number of remaining pairs, $\mathcal{D}=\binom{\mathcal{E}}{2}-\mathcal{W}$. Finally, $\mathcal{E}$ equals half the sum of the degrees and $\mathcal{W}$ is a sum over vertices of a function of $d_v$, so the values of $\mathcal{E}$, $\mathcal{W}$ and $\mathcal{D}$ at a configuration are all computable from its degree sequence alone.
\end{proof}

A comment on the role of $\lambda$ and $\mu$ is in order, expanding the discussion of Sec.~\ref{subsec:noquad}. The coefficients $\kappa_0,\dots,\kappa_3$ of the symmetrized function $\bar f = \kappa_0+\kappa_1\mathcal{E}+\kappa_2\mathcal{W}+\kappa_3\mathcal{D}$ appearing in the proof of Theorem~\ref{thm:no-aux} are constants that may well depend on all four parameters through the construction of $f$, and nothing in the argument prevents this. What the lemma asserts is a statement about functional dependence on the configuration. Whatever the constants are, the value $\bar f(x)$ depends on $x$ only through the three numbers $\mathcal{E}(x)$, $\mathcal{W}(x)$ and $\mathcal{D}(x)$, which are determined by the degree sequence of $x$. Two configurations with identical degree sequences therefore receive identical values of $\bar f$ for every admissible choice of the constants. The conditions involving $\lambda$ and $\mu$ constrain common-neighbor counts, and common-neighbor counts are not functions of the degree sequence, as the pair formed by the Petersen graph and any other cubic graph on ten vertices already shows. This is precisely the information that an invariant quadratic is structurally unable to encode, however its coefficients are tuned.

\subsection{Classification and computational verification}
\label{app:no-aux-class}

The classification of Corollary~\ref{cor:no-aux-class} follows from elementary case analysis. For $k=1$ every $1$-regular graph on $2t$ vertices is a perfect matching, and the perfect matchings are exactly the SRGs with parameters $(2t,1,0,0)$, so $S=R$. Complementation maps the edge configuration $x$ to $\mathbf{1}-x$, an affine substitution that maps polynomials of degree at most two to polynomials of degree at most two and exchanges the SRG family with parameters $(N,k,\lambda,\mu)$ with its complementary family; the case $k=N-2$, the cocktail-party graphs, is therefore equivalent to $k=1$. For $k=2$ and $N=5$ the only $2$-regular graph on five vertices is the pentagon, since a disjoint union of cycles covering five vertices with all cycle lengths at least three is necessarily a single five-cycle, so again $S=R$. For every other feasible quadruple in the range explored in this work, and generically for all primitive families, there exist $k$-regular graphs on $N$ vertices that are not SRGs with the target parameters, as the superexponential growth of the number of $k$-regular graphs \cite{mckay_wormald_1991} makes overwhelmingly likely; whenever such a companion exists, Theorem~\ref{thm:no-aux} forbids the encoding.

The boundary of the classification was verified computationally by exhaustive enumeration over all $2^{\binom{N}{2}}$ edge configurations. For $(5,2,0,1)$ the set $\mathcal{M}(H_{\mathrm{gen}})$ contains $12$ configurations, all labeled pentagons, and for $(6,4,2,4)$ it contains $15$ configurations, all labeled copies of $K_{3\times2}$, confirming $S=R$ in both cases. For $(6,3,0,3)$ it contains $70$ configurations, of which only $10$ are labeled copies of $K_{3,3}$, so $S\subsetneq R$. For this last case we also solved the linear feasibility problem for an exact encoding directly, as an independent check that bypasses the symmetrization argument altogether. The coefficients of a candidate quadratic enter linearly in the conditions $f=0$ on the $10$ solutions and $f\ge 1$ on the remaining $32{,}758$ configurations, so the existence of an exact encoding is a linear program in the $121$ coefficients; it is infeasible, in agreement with Theorem~\ref{thm:no-aux}.

\subsection{Invariant penalties of degree three also fail}
\label{app:no-aux-cubic}

One may ask whether an invariant cubic penalty could already encode the SRG conditions, which would make the quartic $H_{\lambda\mu}$ an overshoot. The orbit-counting logic of Lemma~\ref{lem:invariant-quadratics} extends directly. A monomial of degree at most three is labeled by a set of at most three edge slots, and its $\Sigma_N$-orbit is determined by the isomorphism type of the union of those edges. The orbit sums are therefore the counts of the eight possible types, namely a single edge, two edges sharing a vertex, forming the path $P_3$, two disjoint edges, and, with three edges, the triangle, the path $P_4$, the star $K_{1,3}$, the disjoint union $P_3{\cup}K_2$, and the matching of three disjoint edges. Consequently, two graphs sharing all eight counts cannot be distinguished by any invariant polynomial of degree at most three.

Such pairs exist for the benchmark families of this work. The Petersen graph and the $5$-prism $C_5\square K_2$, both $3$-regular on ten vertices and the second not strongly regular, have identical counts of all eight types, namely $15$, $30$ and $75$ for the types with at most two edges, and $0$, $60$, $10$, $240$ and $145$ for the triangle, $P_4$, $K_{1,3}$, $P_3{\cup}K_2$ and the three-edge matching respectively. Likewise, a single degree-preserving edge switch applied to Paley$(9)$ produces a $4$-regular non-SRG graph on nine vertices whose eight counts coincide with those of Paley$(9)$ itself, namely $18$, $54$, $99$ and then $6$, $144$, $36$, $450$, $180$ in the same order, including in particular the number of triangles. For these families, every relabeling-invariant encoding of degree at most three has spurious global minimizers, and the quartic penalty $H_{\lambda\mu}$ has the minimal degree possible among invariant encodings.

\subsection{Exhaustive study of canonically fixed frames}
\label{app:no-aux-frames}

Theorem~\ref{thm:no-aux} feeds on the relabeling symmetry, and the canonical fixing of Appendix~\ref{app:srg-reduction} exists precisely to destroy that symmetry. It is therefore a legitimate question, not answered by the theorem, whether an exact quadratic encoding might exist on the free variables of a canonically fixed frame. Given the fixed rows, the solution set becomes the set $S'$ of valid completions, and one seeks a quadratic $f$ in the $m'$ free variables with $\mathcal{M}(f)=S'$. Deciding this is computationally clean, because the coefficients of $f$ enter the requirements linearly. One imposes $f=0$ on each completion in $S'$ and $f\ge1$ on each of the $2^{m'}-|S'|$ remaining configurations, the normalization of the gap to unity being free of cost, and determines whether this system of linear constraints on the coefficients is feasible. We settled the question exhaustively for the two smallest frames of this work, certifying each outcome against the full configuration space; infeasibility of any subset of the constraints already implies global infeasibility, while feasibility is claimed only after a scan of all configurations reports no violation.

\paragraph{Petersen frame.}
For $(10,3,0,1)$ the $\lambda=0$ fixing of Sec.~\ref{app:lambda-zero-case} leaves $m'=21$ free variables, and a scan of all $2^{21}=2{,}097{,}152$ completions shows that exactly $|S'|=4$ of them are valid Petersen completions. The linear system is infeasible, so no quadratic in the free variables encodes this frame exactly. Infeasibility comes with a duality certificate of Farkas type, which we verified in exact rational arithmetic and which in this case is simple enough to check by hand. There are four spurious completions $\tilde s_1,\dots,\tilde s_4$, all of them $3$-regular graphs, such that the four solutions $s_1,\dots,s_4$ and the four $\tilde s_i$ satisfy, monomial by monomial up to degree two,
\begin{equation}\label{eq:farkas-cert}
\sum_{i=1}^{4}\varphi(\tilde s_i) \;=\; \sum_{i=1}^{4}\varphi(s_i),
\qquad \varphi(x) = \bigl(1,\;(x_a)_a,\;(x_a x_b)_{a<b}\bigr).
\end{equation}
Applying any quadratic $f$, which is an arbitrary linear combination of the entries of $\varphi$, to both sides gives $\sum_i f(\tilde s_i)=\sum_i f(s_i)$; an exact encoding would make the right-hand side vanish while the left-hand side is at least four times the gap, which is a contradiction. The identity~\eqref{eq:farkas-cert} originates in a recombination structure. The eight configurations agree everywhere except on two blocks of coordinates, an $8$-bit block $\mathcal{A}$ formed by the entries joining rows $2,3$ to columns $6,\dots,9$ and a $4$-bit block $\mathcal{B}$ formed by the entries in $\{6,7\}\times\{8,9\}$, and they factor as
\[
\begin{aligned}
s\colon\quad     & (\sigma_1,\psi_1),\,(\bar\sigma_1,\psi_1),\,(\sigma_3,\psi_2),\,(\bar\sigma_3,\psi_2),\\
\tilde s\colon\quad & (\sigma_3,\psi_1),\,(\bar\sigma_3,\psi_1),\,(\sigma_1,\psi_2),\,(\bar\sigma_1,\psi_2),
\end{aligned}
\]
where $\bar \sigma$ denotes the bitwise complement inside $\mathcal{A}$. Both quadruples contain the same four $\mathcal{A}$-patterns and the same multiset $\{\psi_1,\psi_1,\psi_2,\psi_2\}$ of $\mathcal{B}$-patterns, so all moments internal to either block match trivially, and the cross moments match because $\sigma_1+\bar \sigma_1=\sigma_3+\bar \sigma_3=(1,\dots,1)$, which makes the sums $(\sigma_1+\bar \sigma_1)\otimes \psi$ and $(\sigma_3+\bar \sigma_3)\otimes \psi$ coincide for either $\psi$. No quadratic can undo this pairing.

\paragraph{Paley$(9)$ frame.}
For $(9,4,1,2)$ the general-$\lambda$ fixing of Sec.~\ref{app:general-case} leaves $m'=23$ free variables, with $|S'|=4$ valid completions among the $2^{23}=8{,}388{,}608$ possibilities. Here the linear system turns out to be feasible. There exists an integer quadratic $g$ with only $98$ nonzero coefficients, all lying in $\{-7,\dots,20\}$, such that $g=0$ on the four valid completions and $g\ge3$ on the remaining $8{,}388{,}604$ configurations, a statement verified by an exhaustive scan in integer arithmetic and hence exact. Auxiliary-free encodings of canonically fixed frames can therefore exist.

The two frames together show that the frame-level question is genuinely instance-dependent, and they also show why the Paley$(9)$ example does not yield an encoding method. The feasible quadratic was produced by a computation whose input includes the solution set $S'$ itself, so one must already possess the SRGs in order to formulate, let alone certify, the linear system that outputs $g$. This is the circularity of the Hamming-distance construction of Sec.~\ref{app:no-aux-setup} in a richer guise, and the Petersen frame proves that no uniform recipe can exist. The practical alternatives thus remain exactly the two developed in the main text, namely to stay quadratic in the edge variables and accept that the QUBO output is a candidate set requiring classical certification, as in Sec.~\ref{sec:pipeline}, or to enforce the constraints exactly at the price of $\Theta(N^3)$ auxiliary variables, as in Sec.~\ref{sec:full-qubo}.

\section{Reduction of the number of free parameters in the adjacency matrix of an SRG}
\label{app:srg-reduction}

A naive QUBO encoding of the SRG problem treats every entry of the adjacency matrix as an independent binary variable. The resulting search space is highly redundant, since every relabeling of the vertices produces a binary string that encodes the same graph, generating an $N!$-fold degeneracy in the energy landscape. A substantial fraction of this redundancy can be removed before any optimization is performed, by adopting a canonical labeling that exploits the equitable partition of the vertex set induced by a fixed reference vertex \cite{brouwer_haemers_spectra}. In this appendix we make this reduction explicit at the level of binary edge variables, derive the corresponding free-parameter count for arbitrary SRG parameters, and discuss the additional reduction available in the triangle-free regime $\lambda=0$. The construction is consistent with the fixed-row symmetry-breaking strategy of Sec.~\ref{subsec:fixed-row}, of which it constitutes the natural SRG-aware extension.

\subsection{Three-level structure with respect to a reference vertex}

Let $G$ be a strongly regular graph with parameters $(N,k,\lambda,\mu)$ and adjacency matrix $A$, identified throughout with the binary edge variables $x_{i,j}\in\{0,1\}$ via $A_{i,j}=x_{i,j}$, with $x_{i,j}=x_{j,i}$ and $x_{i,i}=0$. Singling out an arbitrary vertex, which we label $0$, induces the canonical partition
\begin{align}
\mathcal{L}_0 &= \{0\}, \notag\\
\mathcal{L}_1 &= \{1,\ldots,k\} = N(0), \\
\mathcal{L}_2 &= \{k+1,\ldots,N-1\} = V\setminus(\{0\}\cup N(0)), \notag
\end{align}
where $N(0)$ denotes the neighborhood of vertex $0$. The strong regularity of $G$ imposes a rigid structure on the inter-level edges, illustrated in Fig.~\ref{fig:srg-levels}. Any vertex $j\in\mathcal{L}_1$ is adjacent to $0$, so it must share exactly $\lambda$ neighbors with $0$; these necessarily lie in $\mathcal{L}_1$, so $j$ has $\lambda$ neighbors in $\mathcal{L}_1$ and $k-\lambda-1$ in $\mathcal{L}_2$. Conversely, any vertex $\ell\in\mathcal{L}_2$ is non-adjacent to $0$ and must share $\mu$ neighbors with it; these lie in $\mathcal{L}_1$, so $\ell$ has $\mu$ neighbors in $\mathcal{L}_1$ and $k-\mu$ in $\mathcal{L}_2$. Counting in two ways the edges between $\mathcal{L}_1$ and $\mathcal{L}_2$ yields the well-known parameter relation
\begin{equation}
k(k-\lambda-1) = \mu(N-k-1).
\label{eq:srg-param-relation}
\end{equation}

\begin{figure}
    \centering
    \begin{subfigure}[b]{0.45\linewidth}
        \centering
        \includegraphics[width=\linewidth]{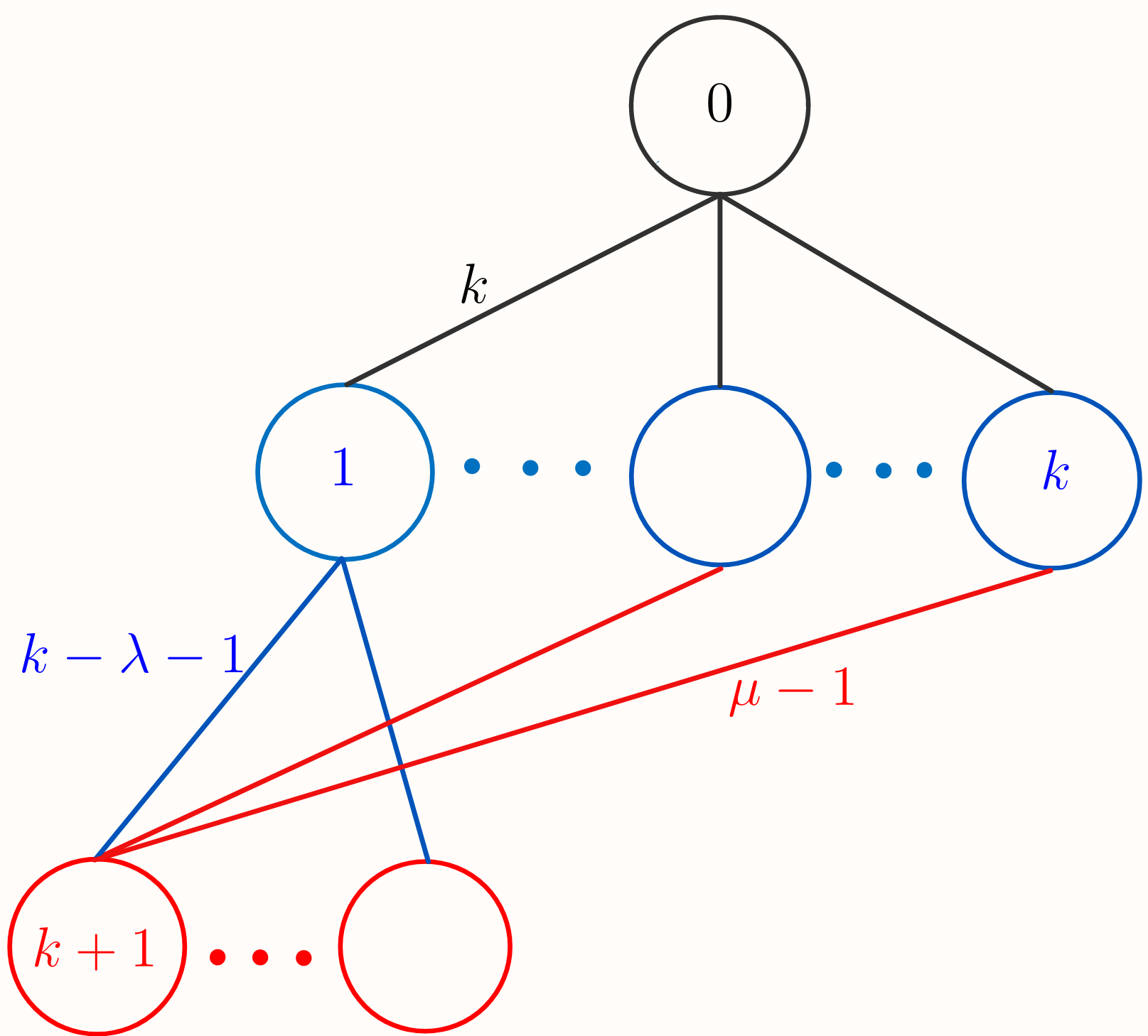}
        \caption{General three-level scheme.}
        \label{fig:srg-levels-general}
    \end{subfigure}
    \hfill
    \begin{subfigure}[b]{0.45\linewidth}
        \centering
        \includegraphics[width=\linewidth]{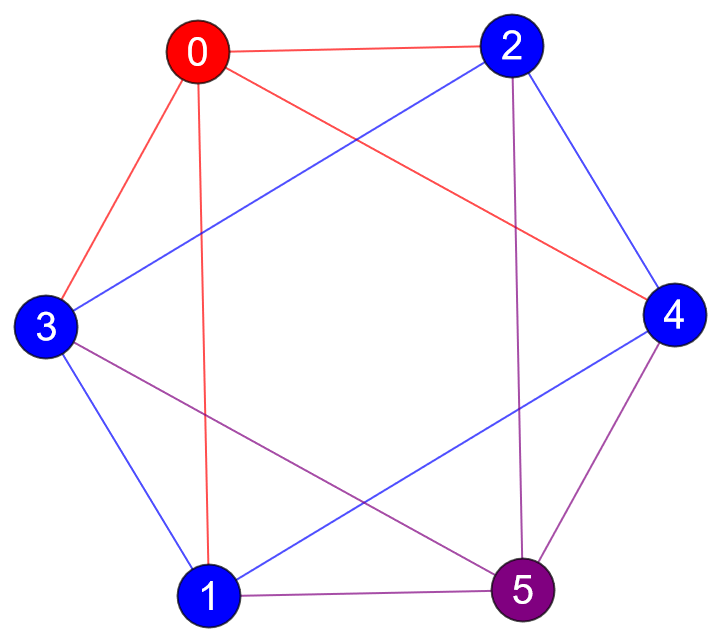}
        \caption{Instance for SRG$(6,4,2,4)$.}
        \label{fig:srg-levels-example}
    \end{subfigure}
    \caption{\justifying Three-level decomposition of the vertex set with respect to a reference vertex $0$. Level $\mathcal{L}_0$ is the singleton $\{0\}$; level $\mathcal{L}_1$ contains the $k$ neighbors of $0$; level $\mathcal{L}_2$ contains the remaining $N-k-1$ vertices.}
    \label{fig:srg-levels}
\end{figure}

\subsection{\texorpdfstring{Canonical fixing for arbitrary $\lambda$}{Canonical fixing for arbitrary lambda}}
\label{app:general-case}

Working in the upper triangle of $A$ leaves $\binom{N}{2}=N(N-1)/2$ binary unknowns once the trivial constraints $x_{i,j}=x_{j,i}$ and $x_{i,i}=0$ are imposed. The freedom to relabel vertices within each level can now be used to fix three rows of the adjacency matrix without loss of generality.

\paragraph*{Row $0$.} By definition of $\mathcal{L}_1$,
\begin{equation}
x_{0,j}=
\begin{cases}
1, & 1\le j\le k,\\
0, & k+1\le j\le N-1,
\end{cases}
\label{eq:row0}
\end{equation}
which fixes all $N-1$ off-diagonal entries of row $0$.

\paragraph*{Row $1$.} Vertex $1\in\mathcal{L}_1$ has exactly one neighbor in $\mathcal{L}_0$, namely vertex $0$, together with $\lambda$ neighbors in $\mathcal{L}_1\setminus\{1\}$ and $k-\lambda-1$ neighbors in $\mathcal{L}_2$. Relabeling vertices within $\mathcal{L}_2$ allows one to take the latter to be the smallest available indices, so
\begin{equation}
x_{1,j}=
\begin{cases}
1, & j=0,\\
\text{free}, & 2\le j\le k,\\
1, & k+1\le j\le 2k-\lambda-1,\\
0, & 2k-\lambda\le j\le N-1.
\end{cases}
\label{eq:row1-general}
\end{equation}
Among the $N-2$ entries with $j\ge 2$, this fixes $N-k-1$ entries and leaves $k-1$ free, subject to the global constraint $\sum_{j=2}^k x_{1,j}=\lambda$ enforced by the SRG conditions.

\paragraph*{Row $k+1$.} Vertex $k+1\in\mathcal{L}_2$ shares $\mu$ neighbors with $0$, all of which lie in $\mathcal{L}_1$. The choice of vertex $1$ as the second anchored vertex in row~\eqref{eq:row1-general} forces $x_{1,k+1}=1$, accounting for one of these $\mu$ neighbors. Relabeling within $\mathcal{L}_1\setminus\{1\}$ allows the remaining $\mu-1$ to be placed at the smallest available indices, yielding
\begin{equation}
x_{k+1,j}=
\begin{cases}
0, & j=0,\\
1, & 1\le j\le \mu,\\
\text{free}, & \mu+1\le j\le k\ \text{or}\ k+2\le j\le N-1.
\end{cases}
\label{eq:rowk1-general}
\end{equation}
Of the entries appearing for the first time, only $\mu-1$ are fixed, namely those with $2\le j\le\mu$; the entry at $j=1$ is already determined by~\eqref{eq:row1-general} and the entry at $j=0$ by~\eqref{eq:row0}.

\begin{observation}
\label{obs:count-general}
Under the canonical fixing~\eqref{eq:row0}--\eqref{eq:rowk1-general}, the upper-triangular adjacency matrix of an SRG$(N,k,\lambda,\mu)$ contains
\begin{eqnarray}
F(N,k,\mu) &\;=\;& \binom{N}{2}-(N-1)-(N-k-1)-(\mu-1)\nonumber\\ &\;=\;& \frac{N^2-5N+2k-2\mu+6}{2}
\label{eq:F-general}
\end{eqnarray}
binary parameters that are not determined by the labeling convention.
\end{observation}

\begin{proof}
The upper triangle of $A$ contains $\binom{N}{2}$ entries. Equation~\eqref{eq:row0} fixes $N-1$ of them. Equation~\eqref{eq:row1-general} contributes $N-2$ entries with $j\ge 2$; among these, $k-1$ remain free and $N-k-1$ are fixed, namely $k-\lambda-1$ ones and $N-2k+\lambda$ zeros. Equation~\eqref{eq:rowk1-general} contributes new fixings only at $j=2,\ldots,\mu$, since $j=0,1$ overlap with the previous rows; this adds $\mu-1$ fixed entries. Subtracting the three contributions from $\binom{N}{2}$ yields~\eqref{eq:F-general}.
\end{proof}

Note that the count~\eqref{eq:F-general} does not depend explicitly on $\lambda$; the $\lambda$-dependence enters only through the global constraint $\sum_{j=2}^{k}x_{1,j}=\lambda$, which restricts the values that the $k-1$ free entries of row $1$ can simultaneously take but does not eliminate any of them as independent variables.

\subsubsection*{\texorpdfstring{Example: SRG$(6,4,2,4)$}{Example: SRG(6,4,2,4)}}

For the cocktail-party graph $\mathrm{SRG}(6,4,2,4)$, the prescription~\eqref{eq:row0}--\eqref{eq:rowk1-general} reduces the adjacency matrix to
\begin{equation}
A \;=\; \begin{pmatrix}
0 & 1 & 1 & 1 & 1 & 0 \\
1 & 0 & x_{1,2} & x_{1,3} & x_{1,4} & 1 \\
1 & x_{1,2} & 0 & x_{2,3} & x_{2,4} & 1 \\
1 & x_{1,3} & x_{2,3} & 0 & x_{3,4} & 1 \\
1 & x_{1,4} & x_{2,4} & x_{3,4} & 0 & 1 \\
0 & 1 & 1 & 1 & 1 & 0
\end{pmatrix},
\label{eq:srg6-matrix}
\end{equation}
which contains $F(6,4,4)=6$ free parameters $\{x_{1,2},x_{1,3},x_{1,4},x_{2,3},x_{2,4},x_{3,4}\}$, in agreement with~\eqref{eq:F-general}.

\subsection{\texorpdfstring{Refinement for triangle-free SRGs with $\lambda=0$}{Refinement for triangle-free SRGs with lambda=0}}
\label{app:lambda-zero-case}

When $\lambda=0$, vertex $1$ shares no neighbors with vertex $0$, and the $k-1$ entries of row $1$ left free in~\eqref{eq:row1-general} are forced to vanish:
\begin{equation}
x_{1,2}=x_{1,3}=\cdots=x_{1,k}=0.
\end{equation}
Similarly, vertex $k+1$ is now adjacent to vertex $1$, and $\lambda=0$ forces $x_{k+1,u}=0$ for every $u\in N(1)\setminus\{0,k+1\}=\{k+2,\ldots,2k-1\}$. The remaining $k-\mu-1$ neighbors of vertex $k+1$ in $\mathcal{L}_2$ can again be placed at the smallest available indices by relabeling. Combining these refinements with~\eqref{eq:row0}--\eqref{eq:rowk1-general} fully determines rows $0$, $1$ and $k+1$ of $A$:
\begin{align}
x_{0,j}&=\begin{cases}1,&1\le j\le k,\\ 0,&k+1\le j\le N-1,\end{cases}\label{eq:row0-l0}\\[2pt]
x_{1,j}&=\begin{cases}1,&j=0,\\ 0,&2\le j\le k,\\ 1,&k+1\le j\le 2k-1,\\ 0,&2k\le j\le N-1,\end{cases}\label{eq:row1-l0}\\[2pt]
x_{k+1,j}&=\begin{cases}0,&j=0,\\ 1,&1\le j\le \mu,\\ 0,&\mu+1\le j\le k,\\ 0,&k+2\le j\le 2k-1,\\ 1,&2k\le j\le 3k-\mu-1,\\ 0,&3k-\mu\le j\le N-1.\end{cases}\label{eq:rowk1-l0}
\end{align}

\begin{observation}
\label{obs:count-lambda0}
Under the canonical fixing~\eqref{eq:row0-l0}--\eqref{eq:rowk1-l0}, the upper-triangular adjacency matrix of an $\mathrm{SRG}(N,k,0,\mu)$ contains
\begin{eqnarray}
F_0(N) &\;=\;& \binom{N}{2}-(N-1)-(N-2)-(N-3)\nonumber\\ &\;=\;& \frac{N^2-7N+12}{2}
\label{eq:F-lambda0}
\end{eqnarray}
free binary parameters. In particular, the count is independent of $k$ and $\mu$.
\end{observation}

\begin{proof}
Row $0$ contributes $N-1$ fixed entries, as before. Row $1$ now has all $N-2$ entries with $j\ge 2$ fixed, since the $k-1$ entries left free in~\eqref{eq:row1-general} are forced to zero by $\lambda=0$. Row $k+1$ contributes two disjoint blocks of fixings in the upper triangle. The first block, of size $N-k-2$, consists of the entries with $j\ge k+2$, all determined by~\eqref{eq:rowk1-l0}. The second block, of size $k-1$, lies in column $k+1$ above row $k+1$ and consists of the entries $(j,k+1)$ for $j=2,\ldots,k$, which by symmetry are determined by~\eqref{eq:rowk1-l0} as $\mu-1$ ones followed by $k-\mu$ zeros. Subtracting all four contributions from $\binom{N}{2}$ yields~\eqref{eq:F-lambda0}.
\end{proof}

The reduction $F\to F_0$ amounts to eliminating $(k-1)+(N-\mu-2)=N+k-\mu-3$ further free parameters with respect to~\eqref{eq:F-general}, namely the $k-1$ free entries of row $1$ and $N-\mu-2$ additional entries of row $k+1$, a reduction that grows linearly with the number of vertices. It is worth emphasizing that the prescription~\eqref{eq:row0-l0}--\eqref{eq:row1-l0} for rows $0$ and $1$ coincides with the fixed-row symmetry-breaking strategy adopted in Sec.~\ref{subsec:fixed-row} for the generation of arbitrary $k$-regular graphs. In that broader context, the choice $x_{1,j}=0$ for $2\le j\le k$ is a labeling convention that prevents triangles involving the anchor vertices from being generated by the solver; in the SRG context studied here, the same prescription becomes a strict consequence of the parameter constraint $\lambda=0$, and the additional row~\eqref{eq:rowk1-l0} provides a further reduction specific to the SRG setting.

\subsubsection*{\texorpdfstring{Example: the Petersen graph SRG$(10,3,0,1)$}{Example: the Petersen graph SRG(10,3,0,1)}}

For the Petersen graph $\mathrm{SRG}(10,3,0,1)$, the prescription~\eqref{eq:row0-l0}--\eqref{eq:rowk1-l0} reduces the adjacency matrix to
\begin{equation*}
A \;=\; \begin{pmatrix}
0 & 1 & 1 & 1 & 0 & 0 & 0 & 0 & 0 & 0 \\
1 & 0 & 0 & 0 & 1 & 1 & 0 & 0 & 0 & 0 \\
1 & 0 & 0 & x_{2,3} & 0 & x_{2,5} & x_{2,6} & x_{2,7} & x_{2,8} & x_{2,9} \\
1 & 0 & x_{2,3} & 0 & 0 & x_{3,5} & x_{3,6} & x_{3,7} & x_{3,8} & x_{3,9} \\
0 & 1 & 0 & 0 & 0 & 0 & 1 & 1 & 0 & 0 \\
0 & 1 & x_{2,5} & x_{3,5} & 0 & 0 & x_{5,6} & x_{5,7} & x_{5,8} & x_{5,9} \\
0 & 0 & x_{2,6} & x_{3,6} & 1 & x_{5,6} & 0 & x_{6,7} & x_{6,8} & x_{6,9} \\
0 & 0 & x_{2,7} & x_{3,7} & 1 & x_{5,7} & x_{6,7} & 0 & x_{7,8} & x_{7,9} \\
0 & 0 & x_{2,8} & x_{3,8} & 0 & x_{5,8} & x_{6,8} & x_{7,8} & 0 & x_{8,9} \\
0 & 0 & x_{2,9} & x_{3,9} & 0 & x_{5,9} & x_{6,9} & x_{7,9} & x_{8,9} & 0
\end{pmatrix},
\end{equation*}
which contains $F_0(10)=21$ free parameters, in agreement with~\eqref{eq:F-lambda0}. The remaining freedom must then be resolved by enforcing the SRG constraints with $\mu=1$, which for this parameter set fixes $A$ uniquely up to isomorphism.

\subsection{Implementation in the auxiliary-variable QUBO}
\label{app:srg-reduction-qubo}

The canonical fixings of Sec.~\ref{app:general-case} and Sec.~\ref{app:lambda-zero-case} translate directly into a substantial preprocessing of the auxiliary-variable QUBO of Sec.~\ref{sec:full-qubo}. Their incorporation is purely substitutional and proceeds in two ordered steps, both of which can be performed once and for all before the QUBO matrix is assembled.

\subsubsection{Step 1: Substitution of edge variables}

Each entry $x_{i,j}$ determined by the canonical fixing, Eqs.~\eqref{eq:row0}--\eqref{eq:rowk1-general} for general $\lambda$ or Eqs.~\eqref{eq:row0-l0}--\eqref{eq:rowk1-l0} when $\lambda=0$, is replaced by its fixed value of $0$ or $1$ in every term of the Hamiltonian. The remaining edge variables, numbering $F(N,k,\mu)$ in the general case and $F_0(N)$ when $\lambda=0$, are kept as binary unknowns.

The free entries of row $1$ in Eq.~\eqref{eq:row1-general} are subject to the global constraint $\sum_{j=2}^{k} x_{1,j} = \lambda$, and the free entries of row $k+1$ in Eq.~\eqref{eq:rowk1-general} are subject to analogous global constraints. We do not impose these as external substitutions, because they cannot be resolved into individual $0/1$ assignments without further combinatorial choices that are not implied by the parameters $(N,k,\lambda,\mu)$ alone; instead, they are enforced automatically by the quadratic penalty $H_{\lambda\mu}^{\mathrm{(q)}}$ of Eq.~\eqref{eq:Hlm-aux} at the level of the Hamiltonian.

\subsubsection{Step 2: Propagation to auxiliary variables}

Each auxiliary variable of the Rosenberg quadratization, $y_{i,j,\ell} = x_{i,\ell}\,x_{j,\ell}$ from Eq.~\eqref{eq:aux-def}, is fully determined whenever both $x_{i,\ell}$ and $x_{j,\ell}$ have been fixed in Step~1. Such auxiliaries are substituted by the product of the corresponding fixed values and are removed from the QUBO together with their Rosenberg penalty term in $H_{\text{aux}}$, Eq.~\eqref{eq:Haux}. Auxiliaries for which only one or neither of the two factors is fixed remain as binary unknowns and continue to be governed by their original Rosenberg penalty.

\subsubsection{Variable counting}

Let $\varepsilon$ denote the number of edge variables remaining after Step~1, and $\alpha$ the number of auxiliary variables remaining after Step~2. The first quantity is given in closed form by Eqs.~\eqref{eq:F-general} and~\eqref{eq:F-lambda0}. The second admits the elementary expression
\begin{equation}
\alpha \;=\; \binom{N}{2}\,(N-2) \;-\; \sum_{\ell=0}^{N-1}\binom{d^{\mathrm{fix}}_\ell}{2},
\label{eq:aux-free-count}
\end{equation}
where $d^{\mathrm{fix}}_\ell$ is the number of edges incident to vertex $\ell$ that are determined by Step~1. The subtracted sum counts the triples $(i,j,\ell)$ with $i<j$ and $\ell \notin \{i,j\}$ for which both $x_{i,\ell}$ and $x_{j,\ell}$ are fixed, and therefore the corresponding $y_{i,j,\ell}$ is determined.

For the Petersen graph $\mathrm{SRG}(10,3,0,1)$, the canonical $\lambda=0$ fixing yields $d^{\mathrm{fix}}_0=d^{\mathrm{fix}}_1=d^{\mathrm{fix}}_4=9$ and $d^{\mathrm{fix}}_\ell=3$ for $\ell\in\{2,3,5,6,7,8,9\}$, so that $\sum_\ell\binom{d^{\mathrm{fix}}_\ell}{2} = 3\binom{9}{2}+7\binom{3}{2} = 3\cdot 36 + 7\cdot 3 = 108 + 21 = 129$, giving $\varepsilon = F_0(10) = 21$ and $\alpha = \binom{10}{2}(10-2) - 129 = 360 - 129 = 231$. The total number of binary variables is reduced from $\binom{10}{2}+\binom{10}{2}(10-2) = 45 + 360 = 405$ to $\varepsilon+\alpha = 21+231 = 252$, a reduction of about $38\%$ achieved without any combinatorial enumeration.

\subsubsection{Stopping criterion}

The canonical fixings of Sec.~\ref{app:general-case}--\ref{app:lambda-zero-case} saturate the symmetry that can be removed by relabeling alone, in the sense that any further fixing would require selecting among combinatorially distinct sub-cases not implied by the parameter quadruple $(N,k,\lambda,\mu)$. For instance, an attempt to fix entries of vertex $2 \in \mathcal{L}_1$ beyond what has already been determined would force a choice of which $\lambda$ neighbours of vertex $2$ in $\mathcal{L}_1$ are also neighbours of vertex $1$, a combinatorial datum that depends on the SRG instance and not solely on its parameters. Such fragmentation rapidly multiplies the number of preprocessing branches and offsets the benefit of any further variable reduction. We therefore stop at the canonical fixing of Sec.~\ref{app:general-case}--\ref{app:lambda-zero-case}, and leave the remaining symmetry to be broken implicitly by the QUBO minimization.

We note that more aggressive reductions are possible at the cost of giving up target-agnosticism. When the automorphism group $\mathrm{Aut}(G_{\mathrm{target}})$ of the intended SRG is known explicitly, one could in principle fix one representative per orbit of $\mathrm{Aut}(G_{\mathrm{target}})$ acting on edge variables, yielding a strictly larger variable reduction than the labeling-based fixing developed here. This target-specific reduction requires advance knowledge of the SRG instance one is looking for, however, and is therefore better suited to verification or enumeration tasks than to the construction task addressed in the present work, in which the QUBO formulation is intended to be agnostic to any particular realization of the parameter quadruple $(N,k,\lambda,\mu)$. The canonical fixing of Sec.~\ref{app:general-case}--\ref{app:lambda-zero-case} is therefore the natural compromise, using only the parameters as input and applying uniformly to all SRGs in a given family at the price of leaving residual labeling symmetries unbroken.

\section{Algebraic feasibility conditions for SRG parameters}
\label{app:feasibility}

The QUBO formulation of Sec.~\ref{sec:full-qubo} takes a parameter quadruple $(N,k,\lambda,\mu)$ as input, and produces by construction a Hamiltonian whose global minimizers are in correspondence with the strongly regular graphs having those parameters. A natural prerequisite is that at least one such graph exist or, more weakly, that the quadruple not be ruled out by simple algebraic arguments. In fact, the vast majority of integer quadruples in any reasonable range are incompatible with strong regularity, and the genuine SRGs are remarkably scarce. The purpose of this appendix is to collect, in self-contained form, the canonical algebraic conditions that any SRG must satisfy, and to apply them to enumerate the parameter sets within $N \le 49$, the regime relevant for the experiments and constructions of the present work \cite{brouwer_haemers_spectra,godsil_royle_algebraic}. The conditions presented below are necessary but not sufficient for the existence of an SRG with the specified parameters. Throughout the appendix we use the notation introduced in Sec.~\ref{sec:noqubo} for the adjacency matrix and its associated objects, and continue to denote the two non-trivial eigenvalues of $A$ by $\theta_2 \ge \theta_3$ as in Sec.~\ref{sec:pipeline}.

\subsection{The parameter equation}
\label{app:param-eq}

The most elementary algebraic constraint is purely combinatorial. By counting in two distinct ways the edges joining the neighborhood of any fixed vertex with its complement in $V$, one arrives at the SRG parameter equation
\begin{equation}
k(k - \lambda - 1) \;=\; (N - k - 1)\,\mu,
\label{eq:param-eq-app}
\end{equation}
which we have already encountered as Eq.~\eqref{eq:srg-param-relation} in the analysis of the three-level decomposition of Appendix~\ref{app:srg-reduction}. As a single Diophantine constraint, \eqref{eq:param-eq-app} is by itself a comparatively weak filter, since many quadruples that satisfy it fail one of the stronger conditions discussed below.

\subsection{Spectral integrality}
\label{app:spectral}

A second family of constraints arises from the spectrum of $A$. The defining SRG conditions can be packaged into the matrix identity
\begin{equation}
A^2 \;=\; k\,\mat{I} \,+\, \lambda\,A \,+\, \mu\,(\mat{J}-\mat{I}-A),
\label{eq:matrix-eq}
\end{equation}
which is Eq.~\eqref{eq:srg-cert} written in an equivalent form, since the $(i,j)$ entry of $A^2$ counts walks of length two from $i$ to $j$, which equals $k$ when $i=j$, $\lambda$ when $i\ne j$ and $i,j$ adjacent, and $\mu$ when $i\ne j$ and $i,j$ non-adjacent.

Because $A$ is $k$-regular, $A\mathbf{1}=k\mathbf{1}$ and $\mat{J}$ commutes with $A$. The two matrices admit a common eigenbasis. The all-ones vector $\mathbf{1}$ is an eigenvector of $A$ with eigenvalue $k$ and of $\mat{J}$ with eigenvalue $N$, and any vector orthogonal to $\mathbf{1}$ has $\mat{J}$-eigenvalue $0$. Restricting \eqref{eq:matrix-eq} to the orthogonal complement of $\mathbf{1}$, every non-trivial eigenvalue $\theta$ of $A$ satisfies the quadratic
\begin{equation}
\theta^2 - (\lambda-\mu)\,\theta - (k-\mu) \;=\; 0,
\label{eq:eigen-quadratic}
\end{equation}
whose roots are the two non-trivial eigenvalues
\begin{equation}
\theta_{2,3} \;=\; \frac{(\lambda-\mu) \pm \sqrt{D}}{2},
\qquad
D \,\equiv\, (\lambda-\mu)^2 + 4(k-\mu),
\label{eq:eigenvalues-app}
\end{equation}
with $D\ge 0$ since $\mu\le k$ for any SRG. This is the same expression already encountered in Sec.~\ref{sec:pipeline}.

Let $\nu_2$ and $\nu_3$ denote the multiplicities of $\theta_2$ and $\theta_3$ as eigenvalues of $A$. Two equations determine them,
\begin{align}
\nu_2 + \nu_3 &\;=\; N - 1, \label{eq:mult-sum} \\
k + \nu_2\,\theta_2 + \nu_3\,\theta_3 &\;=\; 0, \label{eq:mult-trace}
\end{align}
where \eqref{eq:mult-sum} reflects that the eigenspace of $k$ is one-dimensional, and \eqref{eq:mult-trace} expresses $\operatorname{tr}A=0$. Solving the linear system,
\begin{equation}
\nu_2, \nu_3 \;=\; \frac{1}{2}\!\left[(N-1) \,\mp\, \frac{2k + (N-1)(\lambda-\mu)}{\sqrt{D}}\right],
\label{eq:multiplicities}
\end{equation}
where the upper sign corresponds to $\nu_2$, the multiplicity of the larger eigenvalue $\theta_2$, and the lower sign to $\nu_3$, the multiplicity of $\theta_3$.
The multiplicities of the eigenvalues of an integer matrix are themselves non-negative integers. Imposing this on \eqref{eq:multiplicities} separates SRG parameter quadruples into two mutually exclusive families.

\paragraph{Type I: integer eigenvalues.}
If $D$ is a perfect square, both eigenvalues $\theta_{2,3}$ are integers, and $\sqrt{D}$ must additionally divide $2k + (N-1)(\lambda-\mu)$ for $\nu_2$ and $\nu_3$ to be integers. Both multiplicities must be strictly positive. Most SRGs with explicit combinatorial constructions fall in this class, including the Petersen graph, the Clebsch graph, the triangular graphs $T(n)$, and the rook graphs $L_2(n)$.

\paragraph{Type II: irrational eigenvalues, the conference quadruples.}
If $D$ is not a perfect square, \eqref{eq:multiplicities} can still produce integer multiplicities provided that the second term in brackets vanishes,
\begin{equation}
2k + (N-1)(\lambda-\mu) \;=\; 0,
\label{eq:conference-condition}
\end{equation}
which forces $\nu_2=\nu_3=(N-1)/2$ and hence $N$ odd. Substituting \eqref{eq:conference-condition} into \eqref{eq:param-eq-app} and simplifying, the conference parameters
\begin{equation}
N = 4\mu + 1, \qquad k = 2\mu, \qquad \lambda = \mu - 1,
\label{eq:conference-params}
\end{equation}
emerge as the unique solution. The eigenvalues are $\theta_{2,3}=(-1\pm\sqrt{N})/2$. Conference graphs include the Paley graphs $P(q)$ defined for prime powers $q\equiv 1\pmod 4$.

The boundary case in which the conference parameters \eqref{eq:conference-params} are realized for $N$ a perfect square deserves a short comment, since it directly affects the classification of three entries in Table~\ref{tab:feasible-srgs}. When $N=q^2$ for an odd integer $q$, the conference relation $D=N$ still holds but is now a perfect square, so the eigenvalues $\theta_{2,3}=(-1\pm q)/2$ are integers and the quadruple is Type~I despite satisfying \eqref{eq:conference-params}. Within $N\le 49$ this occurs for $N\in\{9,25,49\}$, that is, for the parameter quadruples $(9,4,1,2)$, $(25,12,5,6)$ and $(49,24,11,12)$. We classify all three as Type~I on the basis of their integer eigenvalues, in agreement with Brouwer and Haemers~\cite{brouwer_haemers_spectra}; reserving Type~II for the strictly irrational case keeps the I/II partition disjoint and consistent with the eigenvalue-based criterion used throughout this appendix.

\subsection{Krein conditions}
\label{app:krein}

Beyond integrality, the eigenvalues and their multiplicities of any SRG must satisfy two additional polynomial inequalities, known as the \emph{Krein conditions} \cite{brouwer_haemers_spectra}. They originate from the requirement that certain matrices in the dual basis of the Bose--Mesner algebra of the SRG have non-negative entries, a condition which combines spectral and combinatorial information in a non-trivial way. For SRGs, the conditions reduce to the pair of inequalities
\begin{align}
(\theta_2 + 1)\!\left(k + \theta_2 + 2\theta_2\theta_3\right) &\;\le\; (k + \theta_2)\,(\theta_3 + 1)^2, \label{eq:krein1} \\
(\theta_3 + 1)\!\left(k + \theta_3 + 2\theta_2\theta_3\right) &\;\le\; (k + \theta_3)\,(\theta_2 + 1)^2. \label{eq:krein2}
\end{align}
The Krein conditions are typically inactive for small $N$ but they do rule out parameter sets that pass the parameter equation and the integrality test, and they become essential at larger $N$.

\subsection{The absolute bound}
\label{app:absolute-bound}

A fourth necessary condition arises from a geometric argument due originally to Delsarte \cite{delsarte_1973_thesis} and refined into its modern form by Delsarte, Goethals and Seidel \cite{delsarte_goethals_seidel_1977}. The idea is to embed the vertex set of a putative SRG as a configuration of unit vectors in a Euclidean space, in such a way that the SRG structure is fully encoded by the inner-product geometry of the resulting configuration; classical bounds on the size of such configurations then translate into bounds on the parameters of the SRG.

Concretely, let $E_2 \subset \mathbb{R}^N$ denote the $\nu_2$-dimensional eigenspace of $A$ associated with $\theta_2$, and for each vertex $v\in V$ let $u_v$ denote the orthogonal projection of the standard basis vector $e_v\in\mathbb{R}^N$ onto $E_2$, normalized to unit length. The collection $\{u_v\}_{v\in V}$ is the \emph{Euclidean representation} of the SRG associated with $\theta_2$. A direct calculation in the spectral basis of $A$ shows that the inner products of these vectors take only two distinct off-diagonal values, one for adjacent pairs and one for non-adjacent pairs:
\begin{equation}
\langle u_v, u_w \rangle \;=\;
\begin{cases}
\eta_1, & v\sim w, \\
\eta_2, & v\not\sim w,\ v\ne w,
\end{cases}
\label{eq:two-distance}
\end{equation}
where $\eta_1$ and $\eta_2$ are explicit rational functions of $(N,k,\lambda,\mu)$ and the multiplicities $\nu_2,\nu_3$. A configuration of unit vectors with the property \eqref{eq:two-distance} is known as a \emph{spherical two-distance set} and is a classical object in the geometry of finite point sets.

The fundamental bound on the size of any spherical two-distance set in $\mathbb{R}^{\nu_2}$ is \cite{delsarte_goethals_seidel_1977,godsil_royle_algebraic}
\begin{equation}
\big|\{u_v\}_{v\in V}\big| \;\le\; \frac{\nu_2(\nu_2+3)}{2}.
\label{eq:sphere-2-bound}
\end{equation}
The proof, sketched here only at the level of dimension counting, associates with each vector $u_v$ a real-valued polynomial $F_v$ of degree two on the unit sphere in $\mathbb{R}^{\nu_2}$ that vanishes at all other points of the configuration. These polynomials are linearly independent, and the dimension of the relevant space of polynomials on the sphere is precisely $\nu_2(\nu_2+3)/2$, which corresponds to the harmonic polynomials of degrees $0$, $1$ and $2$ in $\nu_2$ variables, restricted to the sphere. The bound \eqref{eq:sphere-2-bound} is the natural analog of the classical observation that an orthonormal set in $\mathbb{R}^{\nu_2}$ contains at most $\nu_2$ vectors, generalizing the linear-independence count from a one-distance setting to a two-distance setting.

Applying \eqref{eq:sphere-2-bound} to both Euclidean representations of the SRG, that is, to the eigenspace $E_2$ of $\theta_2$ and to the eigenspace $E_3$ of $\theta_3$, one obtains the pair of inequalities
\begin{equation}
N \;\le\; \frac{\nu_2(\nu_2+3)}{2}, \qquad N \;\le\; \frac{\nu_3(\nu_3+3)}{2},
\label{eq:absolute-bound}
\end{equation}
collectively known as the \emph{absolute bound}. The qualifier absolute reflects the fact that \eqref{eq:absolute-bound} bounds $N$ in terms of the multiplicities alone, irrespective of any combinatorial input beyond the spectrum.

The absolute bound is generically inactive for the values of $N$ considered in the present work, since both $\nu_2$ and $\nu_3$ tend to grow with $N$ along the families that admit explicit constructions. It does become decisive for triangle-free SRGs with $\mu = 1$, where the multiplicity $\nu_3$ is small relative to $N$. The classical illustration is the Hoffman--Singleton problem. Combining the parameter equation with spectral integrality leaves only the four feasible triangle-free quadruples $(5,2,0,1)$, $(10,3,0,1)$, $(50,7,0,1)$ and $(3250,57,0,1)$, the first three of which correspond respectively to the pentagon, the Petersen graph, and the Hoffman--Singleton graph. The absolute bound and the Krein conditions are very close to saturation for the last quadruple, but neither rules it out, and the existence of an SRG with parameters $(3250,57,0,1)$ remains an open problem.

\subsection{Necessary, not sufficient}
\label{app:not-sufficient}

The four conditions reviewed above, the parameter equation \eqref{eq:param-eq-app}, the spectral integrality conditions of Sec.~\ref{app:spectral}, the Krein inequalities \eqref{eq:krein1}--\eqref{eq:krein2}, and the absolute bound \eqref{eq:absolute-bound}, constitute the canonical set of algebraic feasibility conditions for SRG parameters. A quadruple $(N,k,\lambda,\mu)$ satisfying all four is said to be \emph{algebraically feasible}, and the comprehensive table of feasible parameters maintained by Brouwer \cite{brouwer_online_table} is the standard reference for tracking the realizability status of each feasible case.

It is essential to emphasize that algebraic feasibility is necessary but not sufficient for the existence of an SRG with the given parameters. A classical illustration of the strength of these conditions is $(28,9,0,4)$, which satisfies the parameter equation and the spectral integrality conditions, with $\theta_{2,3}=(1,-5)$ and $(\nu_2,\nu_3)=(21,6)$, yet violates both the second Krein condition and the absolute bound, and is therefore already excluded above. The four conditions are not exhaustive, however, and a small but growing list of algebraically feasible quadruples is known not to correspond to any actual graph. Within the range of the table presented below, the parameter quadruple $(49,16,3,6)$ furnishes such an instance. It satisfies all four canonical conditions, yet was shown not to correspond to any SRG by Bussemaker, Haemers, Mathon and Wilbrink \cite{bussemaker_haemers_mathon_wilbrink_1989}. A more recent illustration further afield is $(85,14,3,2)$, whose non-existence was established only in 2025 through a detailed enumeration of admissible local structures \cite{shpectorov_zhao_2025}.

A useful but more specialized algebraic obstruction, applicable specifically to conference parameters \eqref{eq:conference-params}, is the Bruck--Ryser--Chowla theorem \cite{brouwer_haemers_spectra}. It requires the number of vertices $N$ to be expressible as a sum of two integer squares, $N = a^2 + b^2$. The conference quadruple $(21,10,4,5)$, for instance, satisfies all four canonical conditions of Sec.~\ref{app:param-eq}--\ref{app:absolute-bound}, but $21 = 3 \cdot 7$ contains the prime $3 \equiv 3 \pmod 4$ with odd multiplicity, so $21$ is not a sum of two squares; therefore no SRG with these parameters exists. The same reasoning rules out the conference quadruple $(33,16,7,8)$.

For the purposes of the present work, what matters is that already within a comparatively small range of $N$, the canonical algebraic conditions reduce the universe of candidate quadruples to a small finite set, and that the existence status of each surviving case must be retrieved from the SRG literature rather than derived by purely algebraic means.

\subsection{\texorpdfstring{Feasible parameter sets in the regime $N \le 49$}{Feasible parameter sets in the regime N <= 49}}
\label{app:table}

Applying the four canonical conditions of Sec.~\ref{app:param-eq}--\ref{app:absolute-bound}, together with the redundancy-removing restriction $k \le \lfloor (N-1)/2 \rfloor$, justified because the complement of an SRG with parameters $(N,k,\lambda,\mu)$ is itself an SRG with parameters $(N,\,N-k-1,\,N-2-2k+\mu,\,N-2k+\lambda)$ and quadruples with $k > (N-1)/2$ are therefore not independent, yields the $33$ quadruples listed in Table~\ref{tab:feasible-srgs}. Of these, $30$ correspond to one or more existing SRGs, marked with \textsf{Y} in the last column; the remaining three, $(21,10,4,5)$, $(33,16,7,8)$ and $(49,16,3,6)$, are algebraically feasible but do not correspond to any SRG, for the reasons indicated in Sec.~\ref{app:not-sufficient}. The example identifications and existence status reported in Table~\ref{tab:feasible-srgs} are taken from the SRG literature \cite{brouwer_haemers_spectra,godsil_royle_algebraic,brouwer_online_table,bussemaker_haemers_mathon_wilbrink_1989} and are not contributions of the present work. They are reproduced here only to make explicit which parameter sets constitute meaningful targets for the QUBO formulations of Sec.~\ref{sec:full-qubo}. The list also illustrates that the existence of more than one non-isomorphic SRG with the same parameters is in general the rule rather than the exception, as the parameter set $(36,15,6,6)$ shows, admitting $32{,}548$ non-isomorphic realizations. The QUBO formulations of Sec.~\ref{sec:full-qubo} do not distinguish among such realizations; their set of global minimizers comprises one configuration for every distinct SRG with the target parameters, up to vertex relabeling.

\begin{table*}[!htbp]
\caption{\justifying Algebraically feasible SRG parameter sets in the range $N \le 49$ with $k \le \lfloor (N-1)/2 \rfloor$. All listed quadruples satisfy the parameter equation \eqref{eq:param-eq-app}, the spectral integrality conditions of Sec.~\ref{app:spectral}, the Krein conditions \eqref{eq:krein1}--\eqref{eq:krein2}, and the absolute bound \eqref{eq:absolute-bound}. The column Exists reports \textsf{Y} if at least one SRG with the given parameters is known to exist, and \textsf{N} if non-existence has been proven. Type~I refers to integer eigenvalues, Type~II to strictly irrational eigenvalues; the conference parameters \eqref{eq:conference-params} with $N$ a perfect square, namely $(9,4,1,2)$, $(25,12,5,6)$ and $(49,24,11,12)$, yield integer eigenvalues and are therefore classified as Type~I. Identifications of existing cases are drawn from the SRG literature and are not derived in this work.}
\label{tab:feasible-srgs}
\centering
\begin{tabular}{cccccc}
\hline
$(N,k,\lambda,\mu)$ & Type & $(\theta_2,\theta_3)$ & $(\nu_2,\nu_3)$ & Examples & Exists \\
\hline
$(5,2,0,1)$       & II & $\frac{-1\pm\sqrt{5}}{2}$    & $(2,2)$    & pentagon $C_5$                                & \textsf{Y} \\
$(9,4,1,2)$       & I  & $(1,-2)$                     & $(4,4)$    & $L_2(3)$, Paley$(9)$                          & \textsf{Y} \\
$(10,3,0,1)$      & I  & $(1,-2)$                     & $(5,4)$    & Petersen graph                                & \textsf{Y} \\
$(13,6,2,3)$      & II & $\frac{-1\pm\sqrt{13}}{2}$   & $(6,6)$    & Paley$(13)$                                   & \textsf{Y} \\
$(15,6,1,3)$      & I  & $(1,-3)$                     & $(9,5)$    & $\mathrm{GQ}(2,2)$, complement of $T(6)$      & \textsf{Y} \\
$(16,5,0,2)$      & I  & $(1,-3)$                     & $(10,5)$   & Clebsch graph                                 & \textsf{Y} \\
$(16,6,2,2)$      & I  & $(2,-2)$                     & $(6,9)$    & $L_2(4)$, Shrikhande graph                    & \textsf{Y} \\
$(17,8,3,4)$      & II & $\frac{-1\pm\sqrt{17}}{2}$   & $(8,8)$    & Paley$(17)$                                   & \textsf{Y} \\
$(21,10,3,6)$     & I  & $(1,-4)$                     & $(14,6)$   & complement of $T(7)$                          & \textsf{Y} \\
$(21,10,4,5)$     & II & $\frac{-1\pm\sqrt{21}}{2}$   & $(10,10)$  & excluded by Bruck--Ryser--Chowla              & \textsf{N} \\
$(21,10,5,4)$     & I  & $(3,-2)$                     & $(6,14)$   & triangular graph $T(7)$                       & \textsf{Y} \\
$(25,8,3,2)$      & I  & $(3,-2)$                     & $(8,16)$   & rook graph $L_2(5)$, unique                   & \textsf{Y} \\
$(25,12,5,6)$     & I  & $(2,-3)$                     & $(12,12)$  & Paley$(25)$ and others                        & \textsf{Y} \\
$(26,10,3,4)$     & I  & $(2,-3)$                     & $(13,12)$  & ten Paulus graphs                             & \textsf{Y} \\
$(27,10,1,5)$     & I  & $(1,-5)$                     & $(20,6)$   & complement of the Schl\"afli graph            & \textsf{Y} \\
$(28,12,6,4)$     & I  & $(4,-2)$                     & $(7,20)$   & $T(8)$, three Chang graphs                    & \textsf{Y} \\
$(29,14,6,7)$     & II & $\frac{-1\pm\sqrt{29}}{2}$   & $(14,14)$  & Paley$(29)$ and others                        & \textsf{Y} \\
$(33,16,7,8)$     & II & $\frac{-1\pm\sqrt{33}}{2}$   & $(16,16)$  & excluded by Bruck--Ryser--Chowla              & \textsf{N} \\
$(35,16,6,8)$     & I  & $(2,-4)$                     & $(20,14)$  & complements of $\mathrm{STS}(15)$ block graphs & \textsf{Y} \\
$(36,10,4,2)$     & I  & $(4,-2)$                     & $(10,25)$  & rook graph $L_2(6)$                           & \textsf{Y} \\
$(36,14,4,6)$     & I  & $(2,-4)$                     & $(21,14)$  & $180$ non-isomorphic SRGs                     & \textsf{Y} \\
$(36,14,7,4)$     & I  & $(5,-2)$                     & $(8,27)$   & triangular graph $T(9)$                       & \textsf{Y} \\
$(36,15,6,6)$     & I  & $(3,-3)$                     & $(15,20)$  & $32{,}548$ non-isomorphic SRGs                & \textsf{Y} \\
$(37,18,8,9)$     & II & $\frac{-1\pm\sqrt{37}}{2}$   & $(18,18)$  & Paley$(37)$ and others                        & \textsf{Y} \\
$(40,12,2,4)$     & I  & $(2,-4)$                     & $(24,15)$  & $\mathrm{GQ}(3,3)$ and others                 & \textsf{Y} \\
$(41,20,9,10)$    & II & $\frac{-1\pm\sqrt{41}}{2}$   & $(20,20)$  & Paley$(41)$ and others                        & \textsf{Y} \\
$(45,12,3,3)$     & I  & $(3,-3)$                     & $(20,24)$  & $78$ non-isomorphic SRGs                      & \textsf{Y} \\
$(45,16,8,4)$     & I  & $(6,-2)$                     & $(9,35)$   & triangular graph $T(10)$                      & \textsf{Y} \\
$(45,22,10,11)$   & II & $\frac{-1\pm\sqrt{45}}{2}$   & $(22,22)$  & Mathon's pseudocyclic SRG; see \cite{brouwer_online_table} & \textsf{Y} \\
$(49,12,5,2)$     & I  & $(5,-2)$                     & $(12,36)$  & rook graph $L_2(7)$                           & \textsf{Y} \\
$(49,16,3,6)$     & I  & $(2,-5)$                     & $(32,16)$  & non-existence proven in \cite{bussemaker_haemers_mathon_wilbrink_1989} & \textsf{N} \\
$(49,18,7,6)$     & I  & $(4,-3)$                     & $(18,30)$  & $L_3(7)$ and others; see \cite{brouwer_online_table} & \textsf{Y} \\
$(49,24,11,12)$   & I  & $(3,-4)$                     & $(24,24)$  & Paley$(49)$ and others                        & \textsf{Y} \\
\hline
\end{tabular}
\end{table*}

\FloatBarrier

\bibliography{references.bib}
\end{document}